\documentclass[11pt,reqno]{amsart}

\usepackage{geometry}

\title{Galois representations, $(\varphi, \Gamma)$-modules and prismatic F-crystals}
\author{Zhiyou Wu}

\usepackage{amsmath,amssymb,fancyhdr,url,amsthm,mathtools,mathrsfs,eufrak,setspace}
\usepackage[pdftex]{graphicx}
\usepackage{tikz-cd}
\usetikzlibrary{cd}
\usepackage[mathscr]{euscript}
\usepackage{tikz}
\usepackage{dsfont}
\usetikzlibrary{matrix}

\usepackage{relsize} 
\usepackage[bbgreekl]{mathbbol} 
\usepackage{amsfonts} 
\DeclareSymbolFontAlphabet{\mathbb}{AMSb} 
\DeclareSymbolFontAlphabet{\mathbbl}{bbold} 
\newcommand{\prism}{{\mathlarger{\mathbbl{\Delta}}}}

\makeatletter
\newcommand*{\rom}[1]{\expandafter\@slowromancap\romannumeral #1@}
\makeatother



\newtheorem{proposition}{Proposition}[section]
\newtheorem{theorem}[proposition]{Theorem}
\newtheorem{example}[proposition]{Example}
\newtheorem{corollary}[proposition]{Corollary}
\newtheorem{definition}[proposition]{Definition}
\newtheorem{remark}[proposition]{Remark}
\newtheorem{lemma}[proposition]{Lemma}

\begin{document}

\maketitle

\begin{abstract}
We prove that both local Galois representations and $(\varphi,\Gamma)$-modules can be recovered from prismatic F-crystals, from which we obtain a new proof of the equivalence of Galois representations and $(\varphi,\Gamma)$-modules.
\end{abstract}

\section{Introduction}
 
The theory of $(\varphi, \Gamma)$-modules was developed by Fontaine (\cite{Fontaine2007}) to study local Galois representations. It plays an important role in the study of families of Galois representations and $p$-adic Langlands correspondence. 

The main result of the theory is that $(\varphi, \Gamma)$-modules are equivalent to Galois representations. The rough idea is to encode the difficult deeply ramified part of Galois theory into complicated rings. In other words, representations of complicated groups with simple coefficients are traded with representations of simple groups but with complicated coefficients. 

The key part of the theory is then the construction of these complicated coefficient rings. It is based on the theory of fields of norms, which is a machine to switch between characteristic 0 and characteristic $p$ worlds. There is another, probably more well-known, theory that serves the same purpose, namely perfectoid fields. Indeed, theory of fields of norms can be viewed as a deperfection of perfectoid fields. 
The coefficient rings appearing in $(\varphi, \Gamma)$-modules are certain infinitesimal lifting of fields of norms along the $p$-direction, in technical terms, they are Cohen rings of fields of norms. 

The insight of this work is to put these rings in a world in which they naturally live, namely the framework of prisms as developed by Bhatt and Scholze (\cite{2019arXiv190508229B}). More precisely, these mysterious rings have natural integral structures which can be viewed as prisms, and the rings themselves are viewed as a structure sheaf on the prismatic site. Then $(\varphi, \Gamma)$-modules are vector bundles (with extra structures) on the prismatic site. This perspective is useful since prismatic sites are very rich. In particular, we can encode infinitesimal lifting of perfectoid fields into the prismatic world as well, which links with Galois representations. We can deduce the classical equivalence of Galois representations and $(\varphi, \Gamma)$-modules from this perspective, namely they are both equivalent to a third, arguably more fundamental object, the prismatic F-crystals. In summary, we have

\begin{theorem}
Let $k$ be a perfect field of characteristic $p$, and $K$ be a finite extension of $W(k)[\frac{1}{p}]$. Then $(\varphi,\Gamma)$-modules over $\bold{A}_K$ are equivalent to prismatic F-crystals in $\mathcal{O}_{\prism}[\frac{1}{I_{\prism}}]^{\wedge}_p$-modules over $(\mathcal{O}_K)_{\prism}$.

Moreover, continuous finite free $\mathbb{Z}_p$-representations of the absolute Galois group $G_K$ are also equivalent to prismatic F-crystals in $\mathcal{O}_{\prism}[\frac{1}{I_{\prism}}]^{\wedge}_p$-modules over $(\mathcal{O}_K)_{\prism}$. 
\end{theorem}

\begin{corollary}
The category of continuous finite free $\mathbb{Z}_p$-representations of the absolute Galois group $G_K$ is equivalent to the category of $(\varphi,\Gamma)$-modules over $\bold{A}_K$. 
\end{corollary}

See the main text for explanations of the notation. 

The equivalence of Galois representations and prismatic F-crystals in $\mathcal{O}_{\prism}[\frac{1}{I_{\prism}}]^{\wedge}_p$-modules is also contained in the work of Bhatt and Scholze \cite{bhatt2021prismatic}. The  proof  for $(\varphi, \Gamma)$-modules follows the same line as the proof for Galois representations.

\subsection*{Acknowledgments}
I would like to thank Peter Scholze for helpful discussions and encouragement. I would like to thank Heng Du, Koji Shimiz, Yu Min and Liang Xiao for discussions on the initial draft. I would like to thank Tong Liu for pointing out a mistake in the early version of the paper. I would also like to  thank the anonymous referee for many suggestions any corrections.  I am  grateful to Max Planck Institute for Mathematics in Bonn for its hospitality and financial support.

\subsection*{Convention}

We follow the notation of \cite{2019arXiv190508229B}. Fixing a prime $p$, a $\delta$-ring is a $\mathbb{Z}_{(p)}$-algebra $R$ equipped with a map $\delta : R \rightarrow R$ such that $\delta(0)=\delta(1)=0$ satisfying \[
\delta(x+y) = \delta(x) +\delta(y) + \frac{x^p + y^p -(x+y)^p}{p}
\]
\[
\delta(xy)= x^p \delta(y)+y^p \delta(x) +p \delta(x) \delta(y)
\]
for any $x,y \in R$. We write 
\[
\phi(x) := x^p +p \delta(x)
\]
which is a ring homomorphism lifting the Frobenius. An element $x $ of a $\delta$-ring $R$ is called distinguished if $\delta(x)$ is a unit. 

Following the notation of Scholze,
for an integral perfectoid ring $R$, we denote by $R^{\flat} := \lim_{\phi} R/p$ 
the tilt of $R$. We can also identify $R^{\flat}$ with $\lim_{x \rightarrow x^p} R$ 
as a multiplicative monoid, then there exists a natural monoid map $R^{\flat} \rightarrow R$ given by 
\[
x=(x_0,x_1, \cdots) \rightarrow x^{\sharp} := x_0,
\]
where $x_{i+1}^p =x_i$, so $(x_0,x_1, \cdots)$ represents an element of $\lim_{x \rightarrow x^p} R$. As a standard notation in $p$-adic Hodge theory, we denote 
\[
A_{\text{inf}}(R) := W(R^{\flat}).
\]
There exists a canonical surjection of rings
\[
\theta: W(R^{\flat}) \rightarrow R
\]
characterized by $\theta([x])= x^{\sharp}$. Moreover, we know that $\text{Ker}(\theta)$ is principal and is generated by any distinguished element in the kernel, see \cite{2019arXiv190508229B} lemma 3.8, lemma 2.33 and lemma 2.24 for example. In particular, if $R$ contains a compatible system of $p$-power roots of unit $\zeta_{p^n}$, then 
\[
\epsilon:= (1, \zeta_p, \zeta_{p^2}, \cdots) \in R^{\flat}
\]
and $1+[\epsilon]^{\frac{1}{p}}+\cdots +[\epsilon]^{\frac{p-1}{p}}$
generates $\text{Ker}(\theta)$. 

We will use the theory of diamonds as developed in \cite{2017arXiv170907343S} and \cite{scholze2020berkeley} in a rudimentary way. Recall that a diamond is a sheaf on the  pro-étale site of characteristic $p$ perfectoid spaces which can be written as the quotient  of a representable sheaf by a  pro-étale equivalence relation. There is a functor from analytic adic spaces over $\text{Spa}(\mathbb{Z}_p)$ into diamonds, which sends
\[
X \longrightarrow X^{\diamond}
\]
where $X^{\diamond}$ is the sheaf whose value on a characteristic $p$ perfectoid space $S$ is the set of untilts $S^{\sharp}$ of $S$ together with a map $S^{\sharp} \rightarrow X$ of adic spaces. When $X=\text{Spa}(R,R^+)$ is affinoid, we sometimes denote 
\[
\text{Spd}(R,R^+) := \text{Spa}(R,R^+)^{\diamond}.
\]

When $A = R\otimes_S R$  and $M$ is an $R$-module, we will write 
\[
M \otimes_R A
\]
when $A$ is viewed as an $R$-algebra with respect to the first factor, i.e. the $R$-algebra structure is $R \rightarrow R\otimes_S R$ is $x \rightarrow x\otimes 1$. Similarly, we write 
\[
A\otimes_R M
\]
when $A$ is equipped with the $R$-algebra structure with respect to the second factor. The convention applies also to other situations when $A$ has two structure maps, such as $A = R \hat{\otimes}_S R $.

\section{Prisms} \label{prismsection}

We recall the basic theory of prisms as developed in \cite{2019arXiv190508229B}, and introduce the primary examples that will be relevant to us.

\subsection{Definitions and Examples}

\begin{definition}
A prism is a pair $(A,I)$, where $A$ is a $\delta$-ring, and $I$ is an ideal of $A$ such that $A$ is derived $(p,I)$-complete, $I$ defines a Cartier divisor on $Spec(A)$, and $p \in I+ \phi(I)A$. The category of prisms has objects the prisms, and the arrows are $\delta$-ring maps preserving the given ideals. 

A prism $(A,I)$ is called bounded if $A/I[p^{\infty}] = A/I[p^n]$ for some $n$. 
\end{definition}

\begin{definition}
Let $X$ be a $p$-adic formal scheme, then the (absolute) prismatic site $X_{\prism}$ of $X$ has objects bounded prisms $(A,I)$ together with a map of formal schemes $Spf(A/I) \rightarrow X$. The arrows are morphisms of prisms preserving the structure map to $X$. An arrow  $(A,I) \rightarrow (B,J)$ 
is a cover if $B$ is $(p,I)$-completely flat over $A$. 

When $X= Spf(R)$ is affine, we simplify the notation by writing $R_{\prism} := X_{\prism}$. 
\end{definition}

\begin{example}
Let $k$ be a perfect field of characteristic $p$, then 
\[
(W(k)[[q-1]], ([p]_q))
\]
is a prism in $W(k)_{\prism}$, where 
$[p]_q := \frac{q^p-1}{q-1}$, 
and the $\delta$-structure is given by the usual $\delta$-structure on $W(k)$ and $\delta(q)=0$.  It is clearly
$(p,q-1)$-complete, which is equivalent to being $(p,[p]_q)$-complete as $[p]_q \equiv p \ \text{mod} (q-1)$ and 
$[p]_q \equiv (q-1)^{p-1} \ \text{mod} \ p$. Moreover,  from 
$[p]_q \equiv p \ \text{mod} (q-1)$
we have 
\[
p = [p]_q + (q-1) \alpha
\]
for some $\alpha \in W(k)[[q-1]]$. Applying $\phi$ to both sides, we have 
\[
p = \phi([p]_q) + (q^p-1) \phi(\alpha) = \phi([p]_q) + [p]_q (q-1) \phi(\alpha),
\]
proving $p \in ([p]_q,\phi([p]_q))$. 
\end{example}

\begin{example}
Let $\mathbb{C}$ be the completion of algebraic closure of $W(k)$, and $\mathcal{O}_{\mathbb{C}}$ 
its ring of integers. Then 
$(W(\mathcal{O}_{\mathbb{C}}^{\flat}), Ker(\theta))$
is a prism, where 
$\theta: W(\mathcal{O}_{\mathbb{C}}^{\flat}) \rightarrow \mathcal{O}_{\mathbb{C}}$
is the canonical map characterized by 
$\theta([\alpha]) = \alpha^{\sharp}$. 
We choose a compatible system $\{\zeta_{p^n}\}$ of $p$-power roots in $\mathbb{C}$, and let 
\[
\epsilon := (1, \zeta_p, \zeta_{p^2}, \cdots) \in \mathcal{O}_{\mathbb{C}}^{\flat},
\]
then it is well-known that 
$1+ [\epsilon^{\frac{1}{p}}] + [\epsilon^{\frac{2}{p}}] + \cdots + [\epsilon^{\frac{p-1}{p}}]$
generates $Ker(\theta)$, and we have a map of $\delta$-rings 
\[
W(k)[[q-1]] \rightarrow  W(\mathcal{O}_{\mathbb{C}}^{\flat}) \]
by sending $q$ to $[\epsilon^{\frac{1}{p}}]$, 
then $[p]_q$ is mapped to 
$1+ [\epsilon^{\frac{1}{p}}] + [\epsilon^{\frac{2}{p}}] + \cdots + [\epsilon^{\frac{p-1}{p}}]$, 
and the condition 
$p \in (Ker(\theta), \phi(Ker(\theta)))$
follows from the same condition for the prism $(W(k)[[q-1]],([p]_q))$. 

Moreover, 
$W(\mathcal{O}_{\mathbb{C}}^{\flat})$
is obviously $(p, [\varpi])$-complete, where $\varpi \in \mathcal{O}_{\mathbb{C}}^{\flat}$
is any nonzero topologically nilpotent element of 
$\mathcal{O}_{\mathbb{C}}^{\flat}$. This implies
$(p, Ker(\theta))$-completeness 
since 
$1+ [\epsilon^{\frac{1}{p}}] + [\epsilon^{\frac{2}{p}}] + \cdots + [\epsilon^{\frac{p-1}{p}}] \
mod \ p$
is such a $\varpi$. 

\end{example}

We made use of an embedding 
$W(k)[[q-1]] \rightarrow  W(\mathcal{O}_{\mathbb{C}}^{\flat})$
sending $q$ to $[\epsilon^{\frac{1}{p}}]$ 
in the previous example. This is not standard, and from now on we view $W(k)[[q-1]]$ as embedded into $W(\mathcal{O}_{\mathbb{C}}^{\flat})$
using the embedding 
\[
W(k)[[q-1]] \rightarrow  W(\mathcal{O}_{\mathbb{C}}^{\flat})
\]
which sends $q$ to $[\epsilon]$. It is the $\phi$-twist of the previous embedding. 

We now recall the theory of fields of norms, see \cite{ASENS_1983_4_16_1_59_0} for details. Let $K$ be a finite totally ramified extension of $W(k)[\frac{1}{p}]$ contained in a fixed completed algebraically closure $\mathbb{C}$ of $W(k)[\frac{1}{p}]$, then we can associate the cyclotomic tower 
\[ K \subset K(\zeta_{p})
\subset K(\zeta_{p^2}) \subset \cdots 
\]
the field of norms 
$\bold{E}_K$, whose ring of integers $\bold{E}_K^+$ can be characterized as a subring of  $\mathcal{O}_{K(\zeta_{p^{\infty}})^{\wedge}}^{\flat}$, namely
\[
\bold{E}_K^+ = \{ (\alpha_n)_n \in  \text{lim} \ \mathcal{O}_{K(\zeta_{p^{\infty}})^{\wedge}}/p \ | \ \ \alpha_n \in \mathcal{O}_{K(\zeta_{p^n})} /p \text{ for n sufficiently large}
\} 
\subset \mathcal{O}_{K(\zeta_{p^{\infty}})^{\wedge}}^{\flat}.
\]
Then $\bold{E}_K^+$ is a complete discrete valuation ring of characteristic $p$, which by construction contains $\bold{E}_{W(k)[\frac{1}{p}]}^+$. Moreover, we know that $\bold{E}_K$ is a finite separable extension of $\bold{E}_{W(k)[\frac{1}{p}]}$. 

We can compute 
$\bold{E}_{W(k)[\frac{1}{p}]}^+$
explicitly as 
\[
\bold{E}_{W(k)[\frac{1}{p}]}^+ = k[[\epsilon-1]] \subset \mathcal{O}_{W(k)[\frac{1}{p}](\zeta_{p^{\infty}})^{\wedge}}^{\flat} \subset \mathcal{O}_{\mathbb{C}}^{\flat}.
\]
We observe that 
\[
W(k)((q-1)) := W(k)[[q-1]][\frac{1}{q-1}]^{\wedge}_p
\subset W(\mathbb{C}^{\flat})
\]
is a Cohen ring of 
$\bold{E}_{W(k)[\frac{1}{p}]}$. By the henselian property of 
$W(k)((q-1))$,
the extension 
\[
\bold{E}_{W(k)[\frac{1}{p}]} \subset \bold{E}_{K} \subset \mathbb{C}^{\flat}
\]
lifts canonically to 
\[
W(k)((q-1)) \subset
\bold{A}_K
\subset
W(\mathbb{C}^{\flat}) 
\]
for a Cohen ring $\bold{A}_K$ of $\bold{E}_K$. Then as $W(k)((q-1)) \subset W(\mathbb{C}^{\flat})$ is invariant under the lift of Frobenius $\phi$ on $W(\mathbb{C}^{\flat})$, so is $\bold{A}_K$   by naturality of $\bold{A}_K$ being extension of  $W(k)((q-1))$ inside $W(\mathbb{C}^{\flat})$. 

Let $\overline{W(k)[[q-1]]}^K$  be the   integral closure of $W(k)[[q-1]]$ in $\bold{A}_K$, and 
\[
\bold{A}_K^+ :=  (\overline{W(k)[[q-1]]}^K)^{\wedge}_p,
\]
i.e. the $p$-adic completion of $\overline{W(k)[[q-1]]}^K$.  We know that 
$W(k)[[q-1]]$ is invariant under the Frobenius, so is $\overline{W(k)[[q-1]]}^K$
\footnote{For $y\in \overline{W(k)[[q-1]]}^K$ satisfying $f(y)=0$ with $f = x^n +a_1x^{n-1} + \cdots + a_n$ a monic polynomial with $W(k)[[q-1]]$ coefficients, $\phi(y)$ satisfies $\phi(y)^n +\phi(a_1)\phi(y)^{n-1} + \cdots + \phi(a_n)=0$.}.
Passing to the $p$-adic completion, this equips  $\bold{A}_K^+$ with a $\delta$-ring structure such that the embedding $\bold{A}_K^+ \subset \bold{A}_K$
is a $\delta$-ring map. The inclusion 
$W(k)[[q-1]] \subset
\bold{A}_K^+$
enables us to view $[p]_q$ as elements of $\bold{A}_K^+ $. We have the following lemmas.

\begin{lemma}
The inclusion $\bold{A}_K^+ \subset \bold{A}_K$ induces an isomorphism
    \[
    \bold{A}_K^+/p \cong \bold{E}_K^+ \subset \bold{E}_K \cong \bold{A}_K/p.
    \]
\end{lemma}

\begin{proof}
    First we show that the natural map 
    $\bold{A}_K^+/p = \overline{W(k)[[q-1]]}^K/p \rightarrow \bold{A}_K/p$ is an inclusion. Suppose that $x\in \overline{W(k)[[q-1]]}^K$ is in the kernel, then 
    $x= py$ for some $y \in \bold{A}_K$.  If $x$ satisfies the minimal ploynomial equation
    $x^n +a_1x^{n-1} + \cdots + a_n=0$ 
    with $a_i \in W(k)[[q-1]]$, then we have that $y$ satisfies the minimal polynomial equation 
    $y^n +\frac{a_1}{p}x^{n-1} + \cdots + \frac{a_n}{p^n}=0$
    with coefficients in the fraction field of the complete discrete valuation ring $W(k)((q-1))$, but $y \in \bold{A}_K$ is integral over $W(k)((q-1))$ as it is an extension of complete DVRs, which tells us that $\frac{a_1}{p}x^{n-1}, \cdots, \frac{a_n}{p^n} \in W(k)((q-1))$. Further, $a_i \in w(k)[[q-1]]$ forces us that 
    $\frac{a_1}{p}x^{n-1}, \cdots, \frac{a_n}{p^n} \in W(k)[[q-1]]$, so $y \in \overline{W(k)[[q-1]]}^K$ proving the injectivity. 

   Next we note that by definition $\bold{A}_K^+/p$ is integral over  $W(k)[[q-1]]/p \cong k[[\epsilon-1]]$, so
\[
   \bold{A}_K^+/p \subset \bold{E}_K^+.
\]
 If $\bar{\alpha}\in \bold{E}_K^+$, and $\bar{f} \in k[[\epsilon-1]][x]$ is the minimal monic polynomial of $\bar{\alpha}$, then $f$ is separable as $\bold{E}_K$ is a separable extension of $k((\epsilon-1))$. We choose a monic lift $f \in W(k)[[q-1]][x]$ of $\bar{f}$, then Hensel's lemma tells us that $f$ has a root $\alpha$ in $\bold{A}_K$ 
that reduces to $\bar{\alpha}$, which means that $\alpha \in \overline{W(k)[[q-1]]}^K$ is a lift of $\bar{\alpha}$, so 
$ \bold{E}_K^+ \subset \bold{A}_K^+/p$.
\end{proof}

\begin{corollary}
    $\bold{A}_K^+$ is a Noetherian ring. 
\end{corollary}
\begin{proof}
    This follows immediately from \cite{stacks-project} tag 05GH.
\end{proof}

\begin{lemma}
For every $n \in \mathbb{N}$, 
$(\bold{A}_K^+,( \phi^n([p]_q) ))$
is a prism. 
\end{lemma}

\begin{proof}
Being a subring of $\bold{A}_K$, 
$\bold{A}_K^+$ is an integral domain, so $ \phi^n([p]_q) $
is a non-zero divisor. Since $(W(k)[[q-1]],([p]_q))$ is a prism, we have  $p= a[p]_q + b\phi([p]_q)$
for some $a,b \in W(k)[[q-1]]$. Applying $\phi^n$ to it, we have 
$p = \phi^n(a) \phi^n([p]_q) + \phi^n(b) \phi^{n+1}([p]_q)$, proving $p \in (\phi^n([p]_q), \phi(\phi^n([p]_q)))$. Lastly, $\bold{A}_K^+$ is $p$-complete by construction and $ \bold{A}_K^+/p \cong \bold{E}_K^+$, which is a complete DVR. We have $\phi^n([p]_q)   \equiv (q-1)^{p^n(p-1)} \ \text{mod}\ p$, which is a pseudouniformizer in 
$\bold{E}_{W(k)[\frac{1}{p}]}^+ \cong k[[q-1]]$, hence a pseudouniformizer in $\bold{E}_K^+$. This proves that $\bold{A}_K^+$ is $(p, \phi^n([p]_q))$-complete, by \cite{stacks-project} tag 0DYC. 
\end{proof}

We record some simple algebraic  properties of the prism $(\bold{A}_K^+,( \phi^n([p]_q) ))$. 

\begin{lemma}
The canonical inclusion
$W(k)[[q-1]] \rightarrow \bold{A}_K^+$
is  faithfully flat. 
\end{lemma}

\begin{proof}
By remark 4.31 of \cite{Bhatt2018}, it is enough to prove flatness of the maps after reduction mod $p$, which is 
$k[[q-1]] \rightarrow \bold{E}^+_K$. This is an injective map from a DVR to an integral domain, hence flat.

To show it is faithfully flat, it is enough to show that for any finitely generated $W(k)[[q-1]]$-module $M$ such that $M \otimes_{W(k)[[q-1]]} \bold{A}_K^+ =0$, then $M=0$. We have 
\[
M/p\otimes_{W(k)[[q-1]]/p} \bold{A}_K^+/p  \cong ( M \otimes_{W(k)[[q-1]]} \bold{A}_K^+ )  \otimes_{\bold{A}_K^+} \bold{A}_K^+/p =0,
\]
but 
$W(k)[[q-1]]/p = k[[q-1]]\rightarrow \bold{E}_K^+=\bold{A}_K^+/p$ 
is faithfully flat as it is a local flat map between DVRs. Thus we have $M/p =0$, which implies that $M=0$ by Nakayama.
\end{proof}

\begin{corollary}
$\bold{A}_K^+/\phi^n([p]_q)$ is 
$p$-torsionfree, so the prism $(\bold{A}_K^+,( \phi^n([p]_q) ))$
is bounded. 
\end{corollary}

\begin{proof}
We have that
\[
W(k)[[q-1]]/\phi^n([p]_q)  \rightarrow \bold{A}_K^+/\phi^n([p]_q) 
\]
is flat since it is the base change of 
$W(k)[[q-1]] \rightarrow \bold{A}_K^+$. We observe that 
\[
W(k)[[q-1]]/\phi^n([p]_q) \cong W(k)[\zeta_{p^n}],
\]
which is $p$-torsionfree. Then the flatness tells us that $\bold{A}_K^+/\phi^n([p]_q)$
is also $p$-torsionfree.
\end{proof}

\begin{lemma} \label{oooooo}
The map $\phi: \bold{A}^+_K \rightarrow \bold{A}^+_K$ is  faithfully flat.
\end{lemma}

\begin{proof}
Again by remark 4.31 of \cite{Bhatt2018}, it is enough to show the flatness mod $p$, which is  the Frobenius  $\phi : \bold{E}^+_K \rightarrow \bold{E}^+_K$. It is an injective map of DVRs, so flat. 

To show it is faithfully flat, it is enough to show that for any finitely generated $\bold{A}_K^+$-module $M$ such that $M \otimes_{\bold{A}_K^+, \phi} \bold{A}_K^+ =0$, then $M=0$. We have 
$M/p\otimes_{\bold{A}_K^+/p, \phi} \bold{A}_K^+/p  \cong ( M \otimes_{\bold{A}_K^+, \phi} \bold{A}_K^+ )  \otimes_{\bold{A}_K^+} \bold{A}_K^+/p =0$, but 
$\phi: \bold{A}_K^+/p = \bold{E}_K^+ \rightarrow \bold{E}_K^+$ is faithfully flat as it is a local flat map between DVRs. Thus we have $M/p =0$, which implies that $M=0$ by Nakayama.
\end{proof}

\subsection{Perfect prisms and perfectoid rings}

There is an important class of prisms that has close connection with perfectoid rings. 

\begin{definition}
A prism $(A,I)$ is called perfect if $\phi$ is an automorphism of $A$. 
\end{definition}

We have a natural perfection functor for prisms. 

\begin{proposition}
(\cite{2019arXiv190508229B} lemma 3.8) Let $(A,I)$ be a prism, and
\[
A_{\text{perf}} := (\underset{\phi}{colim} \ A)^{\wedge}_{(p,I)},
\]
then $(A_{perf},IA_{perf})$ is a perfect prism. Moreover, it is the universal perfect prism over $(A,I)$.
\end{proposition}

Perfect prisms are canonically equivalent to integral perfectoid rings as defined in \cite{Bhatt2018}. We recall the definition of integral perfectoid rings first. 

\begin{definition}
A ring $R$ is integral perfectoid if it is $\pi$-adically complete for some $\pi \in R$ such that $\pi^p$ divides $p$, the Frobenius on $R/p$ is surjective, and the canonical map 
$\theta: W(R^{\flat}) \rightarrow R$ 
has principal kernel. 
\end{definition}

The desired equivalence with perfect prisms is the following theorem. 

\begin{theorem}
(\cite{2019arXiv190508229B} theorem 3.9) The category of perfect prisms is equivalent to the category of integral perfectoid rings. The equivalence functors are $(A,I) \rightarrow A/I$, and $R \rightarrow (W(R^{\flat}), Ker(\theta))$.
\end{theorem}

There is another notion of perfectoid rings used in the theory of perfectoid spaces. We recall the definition and compare it with integral perfectoid rings. Recall that a complete Tate ring is a complete Huber ring that contains a topological nilpotent unit. In more concrete terms, it is a complete topological ring $R$ which contains an open subring $R^+$ whose topology is  $\pi$-adic for some element $\pi \in R^+$, and $R= R^+[\frac{1}{\pi}]$. For any Huber ring $R$, we denote by $R^{\circ}$ the subring of power bounded elements.

\begin{definition}
A perfectoid Tate ring is a uniform complete Tate ring $R$, i.e. $R^{\circ}$ is bounded, such that there exists a topological nilpotent unit $\pi \in R^{\circ}$ such that $\pi^p$ divides $p$, and Frobenius is surjective on $R^{\circ}/\pi^p$. 
\end{definition}

We have the following comparison between the two notions of perfectoid rings. Recall that a ring of integral elements of a Huber ring $R$ is an open and integrally closed subring $R^+$ of $R^{\circ}$. 

\begin{proposition} \label{perfectoidtate}
(\cite{Bhatt2018} lemma 3.20)
Let $R$ be a complete Tate ring, and $R^+ \subset R$ be a ring of integral elements. Then $R$ is a perfectoid Tate ring  if and only if $R^+$ is bounded in $R$ and integral perfectoid.
\end{proposition}

The proposition characterizes integral perfectoid subrings of a perfectoid Tate ring. We can also build a perfectoid Tate ring from a integral perfectoid ring as in the following proposition. 

\begin{proposition} \label{integralperfectoid}
Let $R$ be an integral perfectoid ring and $\pi \in R$ be an element such that $R$ is $\pi$-adically complete and $\pi^p$ divides $p$, then
$R/Ann(\pi)$  is integral perfectoid, and $R[\frac{1}{\pi}]$ is a perfectoid Tate ring with ring of definition $R/Ann(\pi)$.

Similarly, for $R$ integral perfectoid, then   $R/Ann(p)$ is integral perfectoid and $R[\frac{1}{p}]$ is a perfectoid Tate ring with ring of definition $R/Ann_R(p)$.
\end{proposition}

\begin{proof}
By \cite{Bhatt2018} lemma 3.9, there are units $u,v$ of $R$ such that both $\pi u$ and $pv$ has compatible systems of $p$-power roots in $R$. Then by \cite{2004math......9584G}  16.3.69, 
$R/\text{Ann}_R(\pi u)  $, resp. $R/\text{Ann}_R(p v)$,
is integral perfectoid without $\pi$-, resp. $p$-, torsion (the definition of perfectoid in \cite{2004math......9584G} is the same as being integral perfectoid, as \cite{Bhatt2018} remark 3.8 shows). Now 
$R[\frac{1}{\pi}] $, resp. $R[\frac{1}{p}]$,
is a perfectoid Tate ring with ring of definition $R/\text{Ann}_R(\pi u)$, resp. $R/\text{Ann}_R(\pi u)$, by \cite{Bhatt2018} lemma 3.21.  
\end{proof}

\begin{remark}
We have not excluded the zero ring in the proposition. For example, any perfect ring of characteristic $p$ is integral perfectoid with $\pi =0$, then the rings produced in the proposition are all zero. This tells us that in some sense the integral perfectoid rings are more general than being perfectoid Tate. For example, finite fields are integral perfectoid, but can not be nonzero (ring of integers of) perfectoid Tate in any way. 
\end{remark}

We can compute the perfection of the prism $(\bold{A}_K^+, (\phi^n([p]_q)))$.

\begin{lemma} \label{9900000000}
We have 
\[
(\bold{A}_K^+)_{perf} \cong W(\mathcal{O}_{K(\zeta_{p^{\infty}})^{\wedge}}^{\flat}).
\]
\end{lemma}

\begin{proof}
Let $I=(\phi^n([p]_q))$. 
By \cite{2019arXiv190508229B} corollary 2.31,  we have
\[
(\bold{A}_K^+)_{\text{perf}} \cong
W((\bold{A}_K^+)_{\text{perf}} /p)
\cong
W((\underset{\phi}{\text{colim}} \ \bold{A}_K^+)^{\wedge}_{I}/p)
\cong
W((\underset{\phi}{\text{colim}} \ \bold{A}_K^+/p)^{\wedge}_{I})
\cong 
W((\underset{\phi}{\text{colim}} \ \bold{E}_K^+)^{\wedge}_I)
\]
where we use  lemma 10.96.1 (1) of Stacks project, and commutation of colimit with tensoring with $\mathbb{Z}/p$, in the third equality. 
From the theory of fields of norms, we know that
\[
(\underset{\phi}{\text{colim}} \ \bold{E}_K^+)^{\wedge} = \mathcal{O}_{K(\zeta_{p^{\infty}})^{\wedge}}^{\flat},
\]
where the completion is with respect to the natural valuation. We know that $\phi^n([p]_q) \equiv (q-1)^{p^n(p-1)} \  \text{mod} \ p$ 
is a pseudouniformizer in $k[[q-1]] \subset \bold{E}_K^+ $, hence the completion is the same as completion with respect to $I$. Then we have 
\[
(\bold{A}_K^+)_{\text{perf}} \cong 
W(\mathcal{O}_{K(\zeta_{p^{\infty}})^{\wedge}}^{\flat})
\]
as desired.
\end{proof}

\begin{corollary} \label{corollaryii}
The automorphism group of $(\bold{A}_K^+,(\phi^n([p]_{q})))$
in the category $Spf(\mathcal{O}_K)_{\prism}$
is $\Gamma=Gal(K(\zeta_{p^{\infty}})/K)$. 
\end{corollary}

\begin{proof}
Let $\gamma $ be an automorphism of $(\bold{A}_K^+,(\phi^n([p]_{q})))$, then $\gamma$ by definition is a $\delta$-ring morphism, and is continuous with respect to  $(p, \phi^n([p]_q))$-topology, hence $\gamma$ extends to an automorphism of $(\bold{A}_K^+,(\phi^n([p]_{q})))_{\text{perf}}$. By theorem 3.10 of \cite{2019arXiv190508229B}, the automorphism group of  $(\bold{A}_K^+,(\phi^n([p]_{q})))_{\text{perf}}$ 
as an abstract prism
is the same as the automorphism group of the corresponding integral perfectoid ring, which is $\mathcal{O}_{K(\zeta_{p^{\infty}})^{\wedge}}$ by the proposition. The automorphism of $(\bold{A}_K^+,(\phi^n([p]_{q})))_{\text{perf}}$ as objects of $\text{Spf}(\mathcal{O}_K)_{\prism}$ 
is then the $\mathcal{O}_K$-algebra automorphism of $\mathcal{O}_{K(\zeta_{p^{\infty}})^{\wedge}}$, so 
\[
\gamma \in \text{Aut}(\mathcal{O}_{K(\zeta_{p^{\infty}})^{\wedge}}/\mathcal{O}_K) = \Gamma.
\]
But we observe that $\Gamma$ already acts on $(\bold{A}_K^+,(\phi^n([p]_{q})))$. The action on $\bold{A}_K^+$ is clear from its construction, and we need to check that it preserves the ideal $(\phi^n([p]_{q}))$. 

First,
\[
\gamma(q) = q^{\alpha} = (1+q-1)^{\alpha}= \sum_{i=0}^{\infty} \binom{\alpha}{i} (q-1)^i
\]
for some $\alpha \in \mathbb{Z}_p^{\times}$,
where $\binom{\alpha}{i} := \frac{\alpha \cdot (\alpha-1) \cdots (\alpha-i+1)}{i!}$. This follows from the definition $q=[\epsilon]$, and $\alpha$ is the value at $\gamma$ of the cyclotomic character. Then 
\[
\gamma(\phi^n([p]_q)) = \frac{q^{\alpha p^{n+1}}-1}{q^{\alpha p^n}-1}= \frac{q^{p^n}-1}{q^{\alpha p^n}-1}\times \frac{q^{\alpha p^{n+1}}-1}{q^{p^{n+1}}-1} \times \phi^n([p]_q)
\]
and we claim that 
$\frac{q^{p^n}-1}{q^{\alpha p^n}-1}\times \frac{q^{\alpha p^{n+1}}-1}{q^{p^{n+1}}-1}$
is a unit, which is what we want to prove. This follows from the computation \[
\frac{q^{p^n}-1}{q^{\alpha p^n}-1}= (\alpha + \sum_{i=1}^{\infty} \binom{\alpha}{i+1} (q^{p^n}-1)^i)^{-1}=
\alpha^{-1} + (q-1) (\cdots) 
\]
\[
\frac{q^{\alpha p^{n+1}}-1}{q^{p^{n+1}}-1}=
\alpha + \sum_{i=1}^{\infty} \binom{\alpha}{i+1} (q^{p^{n+1}}-1)^i
= \alpha + (q-1)(\cdots) 
\]
which are units in $W(k)[[q-1]]$ as $\alpha \in \mathbb{Z}_p^{\times}. $

Now the corollary follows from the fact that $\bold{A}_K^+ \rightarrow (\bold{A}_K^+)_{\text{perf}}$
is injective ($\phi $ is injective as $\bold{A}_K^+$ is a $\delta$-subring of the perfect $\delta$-ring $W(\mathbb{C}^{\flat})$).
\end{proof}

Moreover, we have the following proposition.

\begin{proposition} \label{bbbb}
There exists $n \in \mathbb{N}$ such that 
\[
(\bold{A}_K^+, (\phi^n([p]_q))) \in (\mathcal{O}_K)_{\prism},
\]
i.e. there exists a map 
$\mathcal{O}_K \rightarrow \bold{A}_K^+/\phi^n([p]_q)$.
\end{proposition}

\begin{remark}
$\bold{A}_K^+$ depends only on $K(\zeta_{p^{\infty}})$. In other words, $\bold{A}_K^+$ cannot see the difference between $K(\zeta_{p^k})$ and $K$, and this is taken care by the choice of $\phi^n([p]_q)$. 
\end{remark}

\begin{proof}
Let $I=([p]_q)$, 
we have by lemma \ref{9900000000}
\[
(\bold{A}_K^+)_{\text{perf}}/I \cong 
W(\mathcal{O}_{K(\zeta_{p^{\infty}})^{\wedge}}^{\flat})/I 
\cong 
\mathcal{O}_{K(\zeta_{p^{\infty}})^{\wedge}},
\]
where the last isomorphism follows from
and $\phi$ being an automorphism and  $\text{Ker}(\theta) = (\phi^{-1}([p]_q))$, for 
$\theta :W(\mathcal{O}_{K(\zeta_{p^{\infty}})^{\wedge}}^{\flat}) \twoheadrightarrow \mathcal{O}_{K(\zeta_{p^{\infty}})^{\wedge}}$. 
On the other hand, we also have
\[
(\bold{A}_K^+)_{\text{perf}}/I \cong
(\underset{\phi}{\text{colim}} \ \bold{A}_K^+)^{\wedge}_{(p,I)}/I
\cong
(\underset{\phi}{\text{colim}} \ \bold{A}_K^+)^{\wedge}_{p}/I
\cong
(\underset{\phi}{\text{colim}} \ \bold{A}_K^+/\phi^i(I))^{\wedge}_p
\]
by definition of $(\bold{A}_K^+)_{\text{perf}}$, hence
\[
(\underset{\phi}{\text{colim}} \ \bold{A}_K^+/\phi^i(I))^{\wedge}_p
\cong
\mathcal{O}_{K(\zeta_{p^{\infty}})^{\wedge}},
\]
and we claim that this implies that
\begin{equation} \label{jjjjj}
\underset{\phi}{\text{colim}} \ \bold{A}_K^+/\phi^i(I) \cong 
\mathcal{O}_{K(\zeta_{p^{\infty}})}.
\end{equation}

Observe that $\bold{A}_K^+/\phi^i(I)$ is integral over $W(k)$, so
$\underset{\phi}{\text{colim}} \ \bold{A}_K^+/\phi^i(I)$
is also integral over $ W(k)$. 
Since there is no non-trivial integral extension of 
$\mathcal{O}_{K(\zeta_{p^{\infty}})}$
in its completion
$\mathcal{O}_{K(\zeta_{p^{\infty}})^{\wedge}}$, 
$\underset{\phi}{\text{colim}} \ \bold{A}_K^+/\phi^i(I)$, 
being integral over a subring of $\mathcal{O}_{K(\zeta_{p^{\infty}})}$, has to be contained in $\mathcal{O}_{K(\zeta_{p^{\infty}})}$. We look at the short exact sequence of $(\underset{\phi}{\text{colim}} \ \bold{A}_K^+/\phi^i(I))$-modules
\[
0 \longrightarrow 
\underset{\phi}{\text{colim}} \ \bold{A}_K^+/\phi^i(I)
\longrightarrow 
\mathcal{O}_{K(\zeta_{p^{\infty}})}
\longrightarrow
\mathcal{O}_{K(\zeta_{p^{\infty}})}/\underset{\phi}{\text{colim}} \ \bold{A}_K^+/\phi^i(I)  =: M
\longrightarrow 0.
\]
From
\[
(\underset{\phi}{\text{colim}} \ \bold{A}_K^+/\phi^i(I))^{\wedge}_p
\cong
\mathcal{O}_{K(\zeta_{p^{\infty}})^{\wedge}}
\]
and Stacks project lemma 10.96.1,
we have $M^{\wedge}_p =0$. The identification after completion also implies that 
$\underset{\phi}{\text{colim}} \ \bold{A}_K^+/\phi^i(I)$
and 
$\mathcal{O}_{K(\zeta_{p^{\infty}})}$
have the same fraction field,  in other words, $M$ is $p$-torsion. Further, we observe that 
$\underset{\phi}{\text{colim}} \ \bold{A}_K^+/\phi^i(I)$ 
contains 
\[
\underset{\phi}{\text{colim}} \ W(k)[[q-1]]/\phi^i([p]_q) \cong 
W(k)[\zeta_{p^{\infty}}]
\]
over which 
$\mathcal{O}_{K(\zeta_{p^{\infty}})}$
is finite. This implies that $M$ is a finitely generated $(\underset{\phi}{\text{colim}} \ \bold{A}_K^+/\phi^i(I))$-module, which,  together with being $p$-torsion, tells us that $M$ is killed by $p^k$ for some $k$, so $M$ is $p$-adically complete. Then we have 
$M=M^{\wedge}_p = 0$, proving the claim.

Now as $\mathcal{O}_K \subset \mathcal{O}_{K(\zeta_{p^{\infty}})}$ is finite over $W(k)$, 
(\ref{jjjjj}) implies that $\mathcal{O}_K$ factorizes through $\mathcal{O}_K \rightarrow \bold{A}_K^+ /\phi^n(I)$ 
for some $n$. Indeed, let $\mathcal{O}_K=W(k)[z]/f(z)$ for some polymonial $f$ with coefficients in $W(k)$. Then $z$, as an element of $\mathcal{O}_{K(\zeta_{p^{\infty}})} = \underset{\phi}{\text{colim}} \ \bold{A}_K^+/\phi^i(I)$, 
lifts to an element $z_1\in \bold{A}_K^+ /\phi^n(I)$ for some $n>0$, and the relation $f(z)=0$ in the colimit forces that $\phi^m(f(z_1))=0$ in the ring $\bold{A}_K^+/\phi^{n+m}(I)$ for some $m$, which implies that $\mathcal{O}_K \rightarrow \bold{A}_K^+/\phi^{n+m}(I)$. Note that we have used the fact that $\mathcal{O}_K \cong \mathcal{O}_K \otimes_{W(k), \phi^n} W(k) \cong W(k)[x]/\phi^n(f)(x)$, where the first isomorphism follows from the perfectness of $W(k)$.
\end{proof}

\begin{remark}
We have used freely in the above proof the fact that there is no algebraic extension of nonarchimedean fields in their completion. More precisely, let $L$ be an algebraic extension of $W(k)[\frac{1}{p}]$, then there is no nontrivial algebraic extension in its completion $L^{\wedge}$. Equivalently, completion induces an equivalence between algebraic extensions $M/L$ of $L$ and algebraic extension  $M^{\wedge}/L^{\wedge}$ of $L^{\wedge}$.  This follows immediately from the fact that $\mathbb{C}^{\text{Gal}(\bar{L}/M)} = M^{\wedge}$, where $\mathbb{C}$ is a completed algebraic closure of $L$. 
\end{remark}

We extract  the  following  lemma from the proof, which will be useful later. 

\begin{lemma} \label{www.oo}
    The map $\mathcal{O}_K \longrightarrow \bold{A}_K^+/\phi^{n}([p]_q)$ is flat. 
\end{lemma}

\begin{proof}
Since $\mathcal{O}_K$ is a valuation ring, it is enough to show that every non-zero element of $\mathcal{O}_K$ is not a zero-divisor of 
$\bold{A}_K^+/\phi^{n}([p]_q)$.
We observe that 
    \[
    \mathcal{O}_K \rightarrow
\mathcal{O}_{K(\zeta_{p^{\infty}})}
\cong 
\underset{\phi}{\text{colim}} \ \bold{A}_K^+/\phi^i([p]_q) 
    \]
    is flat, so if $x\in \mathcal{O}_K \setminus 0$ satisfies $xy=0$ for some $y \in \bold{A}_K^+/\phi^{n}([p]_q)$, then $y$ defines a zero-divisor of $x$ in the colimit, which has to be $0$ by the flatness of $\mathcal{O}_K $ over 
$\mathcal{O}_{K(\zeta_{p^{\infty}})}$. Thus
    $ \phi^k(y)=0$ for  some $k$, which implies that $y=0$ by the faithfully flatness of $\phi$ established in lemma \ref{oooooo}. 
 \end{proof}

\section{ Prismatic F-crystals}

We recall the definition of prismatic F-crystals and make explicit an example that is relevant for us. Recall that we have a natural structure sheaf of $\delta$-rings $\mathcal{O}_{\prism}$ on $X_{\prism}$, together with an ideal sheaf $I_{\prism}$.  

\begin{definition}
Let $X$ be a $p$-adic formal scheme, and $\mathcal{R}$ be  $\mathcal{O}_{\prism}[\frac{1}{I_{\prism}}]^{\wedge}_p$, the $p$-adic completion of the structure sheaf $\mathcal{O}_{\prism}$ with (locally) a generator of $I_{\prism}$ inverted. A prismatic $F$-crystal on $X$ in $\mathcal{R}$-modules is a finite locally free $\mathcal{R}$-module $\mathcal{M}$ over $X_{\prism}$ such that 
\[
\mathcal{M} (A,I) \otimes_{\mathcal{R}(A,I)} \mathcal{R}( B,IB) \cong \mathcal{M} (B,IB)
\]
for any arrow $(A,I) \rightarrow (B,IB)$ in $X_{\prism}$, together with an isomorphism
\[
F: \phi^*\mathcal{M} [\frac{1}{I_{\prism}}] \cong \mathcal{M}[\frac{1}{I_{\prism}}]
\]
of $\mathcal{R}$-modules.
\end{definition}

A concrete way to work with prismatic F-crystals in $\mathcal{O}_{\prism}[\frac{1}{I_{\prism}}]^{\wedge}_p$-modules is to choose a cover $(A,I)$ of the final object $*$ of the topos $X_{\prism}$, suppose that $(A,I) \times_{*} (A,I)$ is representable, and
\[
(B,J):= (A,I) \times_{*} (A,I),
\]
 then a prismatic F-crystal in $\mathcal{O}_{\prism}[\frac{1}{I_{\prism}}]^{\wedge}_p$-modules is a finite projective $\varphi$-$A[\frac{1}{I}]^{\wedge}_p$-module $M$ together with an isomorphism 
\[
\beta: 
M \otimes_{A[\frac{1}{I}]^{\wedge}_p} B[\frac{1}{J}]^{\wedge}_p \cong B[\frac{1}{J}]^{\wedge}_p \otimes_{A[\frac{1}{I}]^{\wedge}_p} M
\]
of $\varphi$-$B[\frac{1}{J}]^{\wedge}_p$-modules
satisfying cocycle conditions. Indeed, it is obvious that we can obtain such an object from a prismatic $F$-crystal. Conversely, if we are given such data, we can build a prismatic $F$-crystal $\mathcal{M}$ as follows. Given a prism $(C,K) \in X_{\prism}$, since $(A,I)$ covers the final object, there exists a cover $(C',K')$ of $(C,K)$, which also lies over $(A,I)$. Then we can define
\[
\mathcal{M}(C,K) := \text{Eq}( M \otimes_{A[\frac{1}{I}]^{\wedge}_p} C'[\frac{1}{K'}]^{\wedge}_p  
\rightrightarrows
 M \otimes_{A[\frac{1}{I}]^{\wedge}_p} C''[\frac{1}{K''}]^{\wedge}_p  )
\]
where 
$(C'',K'') := (C',K') \times_{(C,K)} (C',K'),$
which lives over $(B,J)$, and the two arrows comes from the base change descent data
$\beta \otimes_{B[\frac{1}{J}]^{\wedge}_p} C''[\frac{1}{K''}]^{\wedge}_p $. Now by 
the next proposition, the descent data is effective on finite projective modules, so $\mathcal{M}(C,K)$ is a finite projective module over $C[\frac{1}{K}]^{\wedge}_p$. 

\begin{proposition}
Let $(A,I)\rightarrow (B,IB)$ be a cover in the category of bounded prisms, and $(B^{\bullet}, I B^{\bullet})$
be the corresponding Čech nerve, then we have an equivalence of categories
\[
\text{Vect}(A[\frac{1}{I}]^{\wedge}_p) \overset{\sim}{\rightarrow} \underset{\leftarrow}{\lim}(Vect(B[\frac{1}{I}]^{\wedge}_p) \rightrightarrows Vect(B^2[\frac{1}{I}]^{\wedge}_p) \substack{\longrightarrow \\[-1em] \longrightarrow \\[-1em] \longrightarrow} \cdots),
\]
where Vect(R) denotes the category of finite projective modules over the ring $R$.  
\end{proposition}

\begin{proof}
Since $A[\frac{1}{I}]^{\wedge}_p$ (resp. $B^{\bullet}[\frac{1}{I}]^{\wedge}_p$) is $p$-complete, by \cite{stacks-project}  tag 0D4B, $Vect(A[\frac{1}{I}]^{\wedge}_p)$ (resp. $ Vect(B^{\bullet}[\frac{1}{I}]^{\wedge}_p$)) is equivalent to $\underset{n}{\lim} \ Vect(A[\frac{1}{I}]/p^n)$ (resp. $\underset{n}{\lim} \ Vect(B^{\bullet}[\frac{1}{I}]/p^n)$). Thus it suffices to prove the equivalence
\begin{equation} \label{uroieuioreurioe}
\text{Vect}(A/p^n[\frac{1}{I}]) \overset{\sim}{\rightarrow} \underset{\leftarrow}{\lim}(Vect(B/p^n[\frac{1}{I}]) \rightrightarrows Vect(B^2/p^n[\frac{1}{I}]) \substack{\longrightarrow \\[-1em] \longrightarrow \\[-1em] \longrightarrow} \cdots).
\end{equation}

We know from the proof of  \cite{2019arXiv190508229B}
corollary 3.12 that 
$B^2$ is the derived $(p,I)$-completion of $B\otimes_A^{\mathbb{L}}B$, which is proved in $loc. \ cit.$ to be discrete and classically $(p,I)$-complete. Then lemma \ref{xieior} shows that 
\[
B^2 = (B\otimes_A B)^{\wedge}_{(p,I)},
\]
i.e. it is the classical $(p,I)$-completion of $B\otimes_A B$, and similarly for $B^{\bullet}$. Indeed, by the discreteness and classical $(p,I)$-completeness of $\widehat{B\otimes_A^{\mathbb{L}}B}$, the derived $(p,I)$-completion of $B\otimes_A^{\mathbb{L}}B$, we have 
\[
B^2 =
H^0(\widehat{B\otimes_A^{\mathbb{L}}B})= H^0(\widehat{B\otimes_A^{\mathbb{L}}B})^{\wedge}_{(p,I)} = H^0(B\otimes_A^{\mathbb{L}}B)^{\wedge}_{(p,I)} 
=(B\otimes_A B)^{\wedge}_{(p,I)},
\]
where the third equality follows from lemma \ref{xieior}.
Then we have 
\[
B^2/p^n = (B/p^n \otimes_{A/p^n} B/p^n)^{\wedge}_{I}
\]
\[
B^3/p^n = (B/p^n \otimes_{A/p^n} B/p^n \otimes_{A/p^n} B/p^n)^{\wedge}_{I}
\]
and similarly for $B^{\bullet}/p^n$. Since $A \rightarrow B$ is assumed to be $(p,I)$-completely faithfully flat, $A/p^n \rightarrow B/p^n$ is $I$-completely faithfully flat, and (\ref{uroieuioreurioe})
follows from \cite{2019arXiv191210968M}
theorem 7.8.
\end{proof}

\begin{lemma} \label{xieior}
Let $R$ be a ring and $I$ be a finitely generated ideal of $R$. Let $K \in D^{\leq 0}(R)$
be a complex of $R$-modules with zero cohomology in positive degrees, and $\hat{K}$ its derived $I$-completion, then we have a canonical identification
\[
H^0(\hat{K})^{\wedge}_I \cong H^0(K)^{\wedge}_I,
\]
in other words, the classical $I$-completion of $H^0(\hat{K})$ is the same as the classical $I$-completion of $H^0(K)$. 

\end{lemma}

\begin{proof}
For any $R$-module $M$, we denote $\hat{M}$ its derived $I$-completion by viewing $M$ as a complex concentrated in zero degree. Recall that there is a natural map $M \rightarrow H^0(\hat{M})$ which is initial among maps $M \rightarrow N$ with $N$ a
derived $I$-complete $R$-module, see \cite{2019arXiv191210968M} definition 2.27 for example. By the spectral sequence in \cite{stacks-project}
tag 0BKE, we have 
\[
H^0(\hat{K})= H^0(\widehat{H^0(K)}),
\]
so we see that
\[
H^0(K)\rightarrow H^0(\hat{K})
\]
is the initial map among $H^0(K) \rightarrow N$
with $N$ a derived $I$-complete $R$-module. We now claim that the composite map
\[
H^0(K) \rightarrow H^0(\hat{K}) \rightarrow H^0(\hat{K})^{\wedge}_I
\]
is initial among $H^0(K) \rightarrow N'$
where $N'$ is classical $I$-complete. This is exactly the universal property of the classical $I$-completion of $H^0(K)$, whence $H^0(\hat{K})^{\wedge}_I = H^0(K)^{\wedge}_I$. 

We now prove the claim. Let $H^0(K) \rightarrow N'$
where $N'$ be as given, since classical $I$-completeness implies derived $I$-completeness (\cite{stacks-project} tag 091T), we have a unique factorization
\[
H^0(K)\rightarrow H^0(\hat{K}) \rightarrow N',
\]
which further factors as
\[
H^0(K)\rightarrow H^0(\hat{K}) \rightarrow H^0(\hat{K})^{\wedge}_I
\rightarrow N'
\]
by the universal property of $H^0(\hat{K})^{\wedge}_I$. 
\end{proof}

\begin{example} \label{examplewww}
Let $k$ be a perfect field of characteristic $p$, $K$ be a finite totally ramified extension of $W(k)[\frac{1}{p}]$, and $X= \text{Spf}(\mathcal{O}_K)$. Let $n \in \mathbb{N}$ be chosen as in proposition \ref{bbbb} such that   
\[
(\bold{A}_K^+,(\phi^n([p]_q)))\in X_{\prism},
\]
then it is a cover of the final object by the following lemma \ref{coverfinal}. Moreover, we have by lemma \ref{89999}
\[
(\bold{A}_K^+, (\phi^n([p]_q))) \times_{\ast} (\bold{A}_K^+, (\phi^n([p]_q))) = (C,J),
\]
where $C$ is $(p,\phi^n([p]_q) \otimes 1)$-completion of 
$\bold{A}_K^+ \hat{\otimes}_{W(k)} \bold{A}_K^+
 \{\frac{\omega \otimes 1- 1\otimes\omega}{\phi^n([p]_q) \otimes 1}, \frac{1\otimes\phi^n([p]_q)}{\phi^n([p]_q) \otimes 1} \}$,
which is  freely adjoining $\frac{\omega \otimes 1- 1\otimes\omega}{\phi^n([p]_q) \otimes 1}$
and 
$\frac{1\otimes\phi^n([p]_q)}{\phi^n([p]_q) \otimes 1} $
to 
$\bold{A}_K^+ \hat{\otimes}_{W(k)} \bold{A}_K^+$
as $\delta$-rings, and $J=(\phi^n([p]_q)\otimes 1)$.

Thus a prismatic F-crystal in 
$\mathcal{O}_{\prism}[\frac{1}{I_{\prism}}]^{\wedge}_p$-modules
over $X$ is a $\bold{A}_K$-module $M$ together with an isomorphism
\[
M\otimes_{\bold{A}_K} C [\frac{1}{J}]^{\wedge}_p
\cong 
C [\frac{1}{J}]^{\wedge}_p 
\otimes_{\bold{A}_K} M
\]
satisfying cocycle conditions. The ring $C$ is difficult to understand explicitly, we will see below how we can bypass this difficulty by passing to perfections. 
\end{example}

\begin{lemma}\label{coverfinal}
Let $K$ be a finite totally ramified extension of $W(k)[\frac{1}{p}]$,  $X= Spf(\mathcal{O}_K)$, and $n \in \mathbb{N}$ be chosen as in proposition \ref{bbbb} such that   
\[
(\bold{A}_K^+,(\phi^n([p]_q)))\in X_{\prism},
\]
then it
covers the final object in $X_{\prism}$. 
\end{lemma}

\begin{proof}
Let $(A,I)$ be an object of $X_{\prism}$, then $A/I$ is a $\mathcal{O}_K$-algebra. We have a quasisyntomic cover $\mathcal{O}_K[\zeta_{p^{\infty}}]^{\wedge}_p$ 
of $\mathcal{O}_K$, hence 
\[
A/I \hat{\otimes}_{\mathcal{O}_K} \mathcal{O}_K[\zeta_{p^{\infty}}]^{\wedge}_p
\]
is a quasisyntomic cover of $A/I$. By \cite{2019arXiv190508229B} proposition 7.11, we can find a prism $(C,J)$ that covers $(A,I)$ such that there is a morphism 
\[
A/I \hat{\otimes}_{\mathcal{O}_K} \mathcal{O}_K[\zeta_{p^{\infty}}]^{\wedge}_p 
\rightarrow 
C/J.
\]
Now as $\mathcal{O}_K[\zeta_{p^{\infty}}]^{\wedge}_p$ is integral perfectoid, the composition 
\[
\mathcal{O}_K[\zeta_{p^{\infty}}]^{\wedge}_p
\rightarrow
A/I \hat{\otimes}_{\mathcal{O}_K} \mathcal{O}_K[\zeta_{p^{\infty}}]^{\wedge}_p 
\rightarrow 
C/J
\]
lifts to a map of prisms
\[
(A_{\text{inf}}(\mathcal{O}_K[\zeta_{p^{\infty}}]^{\wedge}_p), \text{Ker}(\theta)) \rightarrow (C,J)
\]
by \cite{2019arXiv190508229B} lemma 4.7. We have that 
\[
(\bold{A}_K^+,(\phi^n([p]_{q})))_{\text{perf}}= (A_{\text{inf}}(\mathcal{O}_K[\zeta_{p^{\infty}}]^{\wedge}_p), \phi^{n+1}(\text{Ker}(\theta)))
\]
by lemma \ref{9900000000}, so we have a map 
\[
(\bold{A}_K^+,(\phi^n([p]_{q}))) \rightarrow
(A_{\text{inf}}(\mathcal{O}_K[\zeta_{p^{\infty}}]^{\wedge}_p), \phi^{n+1}(\text{Ker}(\theta)))
\overset{\phi^{-n-1}}{\longrightarrow}
(A_{\text{inf}}(\mathcal{O}_K[\zeta_{p^{\infty}}]^{\wedge}_p), \text{Ker}(\theta)) \rightarrow (C,J)
\]
of prisms. As $(C,J)$ covers $(A,I)$, we have finished the proof.
\end{proof}

\begin{lemma} \label{89999}
Let $k$ be a perfect field of characteristic $p$, $K$ be a finite totally ramified extension of $W(k)[\frac{1}{p}]$, and
$\omega \in \bold{A}_K^+ $
be an element which modulo $\phi^n([p]_q)$ becomes a uniformizer of $\mathcal{O}_K$ under the inclusion 
$\mathcal{O}_K \subset \bold{A}_K^+/\phi^n([p]_q)$.
Then as objects of $Spf(\mathcal{O}_K)_{\prism}$, we have 
\[
(\bold{A}_K^+, (\phi^n([p]_q))) \times_{\ast} (\bold{A}_K^+, (\phi^n([p]_q))) = (C,J),
\]
where
\[
(C,J):= 
(\bold{A}_K^+ \hat{\otimes}_{W(k)} \bold{A}_K^+
 \{\frac{\omega \otimes 1- 1\otimes\omega}{\phi^n([p]_q) \otimes 1}, \frac{1\otimes\phi^n([p]_q)}{\phi^n([p]_q) \otimes 1} \}^{\wedge}, (\phi^n([p]_q) \otimes 1)).
\]
The ring displayed
is $(p,\phi^n([p]_q) \otimes 1)$-completion of 
$\bold{A}_K^+ \hat{\otimes}_{W(k)} \bold{A}_K^+
 \{\frac{\omega \otimes 1- 1\otimes\omega}{\phi^n([p]_q) \otimes 1}, \frac{1\otimes\phi^n([p]_q)}{\phi^n([p]_q) \otimes 1} \}$,
which is  freely adjoining $\frac{\omega \otimes 1- 1\otimes\omega}{\phi^n([p]_q) \otimes 1}$
and 
$\frac{1\otimes\phi^n([p]_q)}{\phi^n([p]_q) \otimes 1} $
to 
$\bold{A}_K^+ \hat{\otimes}_{W(k)} \bold{A}_K^+$
as $\delta$-rings.
\end{lemma}

\begin{proof}
By definition, an object of $X_{\prism}$ is a prism $(A,I)$ equipped with a morphism $\mathcal{O}_K \rightarrow A/I$. 
The $\mathcal{O}_K$-algebra structure does not necessarily lift to $A$, but the corresponding $W(k)$-algebra structure does. Indeed,
by deformation theory of perfect rings, $W(k)\subset \mathcal{O}_K \rightarrow A/I$ lifts canonically to a morphism $W(k) \rightarrow A$, hence objects of $X_{\prism}$ are naturally equipped with $W(k)$-algebra structures, and all arrows in $X_{\prism}$ are $W(k)$-algebra morphisms. 

Now given such an $(A,I)$, together with two arrows
$(\bold{A}_K^+, (\phi^n([p]_q))) \rightarrow (A,I)$
in $X_{\prism}$,
they give rise canonically to a $\delta$-ring morphism 
\[
f:
\bold{A}_K^+ \hat{\otimes}_{W(k)} \bold{A}_K^+ \longrightarrow A.
\]
We know from properties of prisms that $I = (f(\phi^n([p]_q) \otimes 1)) = (f(1 \otimes \phi^n([p]_q)))$. Moreover, being arrows in $X_{\prism}$, the two arrows are $\mathcal{O}_K$-algebra morphisms $\text{mod} \ \phi^n([p]_q)$,
which means that $f$ maps
$1\otimes\omega-\omega \otimes 1 \ \text{mod} \ \phi^n([p]_q) \otimes 1$ 
to 0 in $A/I$. Thus
 $f$ factors through 
 $\bold{A}_K^+ \hat{\otimes}_{W(k)} \bold{A}_K^+
 \{\frac{\omega \otimes 1- 1\otimes\omega}{\phi^n([p]_q) \otimes 1}, \frac{1\otimes\phi^n([p]_q)}{\phi^n([p]_q) \otimes 1} \}^{\wedge}$. 

It remains to show that
$
(C,J)
$ 
is a prism. First we observe that 
$\frac{1\otimes\phi^n([p]_q)}{\phi^n([p]_q) \otimes 1}$
is a unit in 
$
C
$, 
whence $(1\otimes\phi^n([p]_q)) = (\phi^n([p]_q) \otimes 1)$. 
This follows from the equation 
\[
1\otimes\phi^n([p]_q)= \phi^n([p]_q) \otimes 1 \cdot \frac{1\otimes\phi^n([p]_q)}{\phi^n([p]_q) \otimes 1}
\]
and \cite{2019arXiv190508229B} lemma 2.24.

We now claim that 
$\{ 1\otimes\phi^n([p]_q),\omega \otimes 1- 1\otimes\omega\}$
is a regular sequence of 
$\bold{A}_K^+ \hat{\otimes}_{W(k)} \bold{A}_K^+$.
Viewing 
$\bold{A}_K^+ \hat{\otimes}_{W(k)} \bold{A}_K^+$ as a $\bold{A}_K^+$-algebra along the first factor, then it follows from \cite{2019arXiv190508229B} proposition 3.13 that 
$
(C,J)
$
is a prism. It remains to prove the claim, which puts us into the setting of \cite{2019arXiv190508229B} proposition 3.13. 

Observe that $1\otimes\phi^n([p]_q)$ is regular since $\phi^n([p]_q)$ is regular in $\bold{A}_K^+$ (being an integral domain), and 
$\bold{A}_K^+ \hat{\otimes}_{W(k)} \bold{A}_K^+ $
is flat over $\bold{A}_K^+$. We want to show that $\omega \otimes 1- 1\otimes\omega$ 
is regular in 
\[
\bold{A}_K^+ \hat{\otimes}_{W(k)} \bold{A}_K^+ / 1\otimes\phi^n([p]_q)
\cong \bold{A}_K^+ \hat{\otimes}_{W(k)} (\bold{A}_K^+ /\phi^n([p]_q)).
\]
We note that 
\[
\omega \otimes 1- 1\otimes\omega \in 
\bold{A}_K^+ \hat{\otimes}_{W(k)} \mathcal{O}_K \subset
\bold{A}_K^+ \hat{\otimes}_{W(k)} (\bold{A}_K^+ /\phi^n([p]_q))
\]
by definition of $\omega$. By our assumption on $k$, $\mathcal{O}_K$ is totally ramified over $W(k)$, so $\mathcal{O}_K \cong W(k)[x]/E$
for an Eisenstein polynomial $E \in W(k)[x]$. Then 
\[
\bold{A}_K^+ \hat{\otimes}_{W(k)} \mathcal{O}_K \cong
\bold{A}_K^+[x]/E,
\]
and by Eisenstein criterion, $E$ is irreducible in 
$\bold{A}_K^+[x]$. Note that $\bold{A}_K^+$ is a regular local ring, whence a UFD by Auslander–Buchsbaum theorem. Then Gauss's lemma tells us $\bold{A}_K^+[x]$ is a UFD as well, so irreducible polynomials are prime. Then  
$\bold{A}_K^+[x]/E$ is an integral domain, and $\omega \otimes 1- 1\otimes\omega$
is regular in $\bold{A}_K^+ \hat{\otimes}_{W(k)} \mathcal{O}_K$. 
We know from lemma  \ref{www.oo} that 
$\mathcal{O}_K \rightarrow \bold{A}_K^+ /\phi^n([p]_q)$
is flat,  so
 $\omega \otimes 1- 1\otimes\omega$
 is regular in 
$\bold{A}_K^+ \hat{\otimes}_{W(k)} (\bold{A}_K^+ /\phi^n([p]_q))$. 
\end{proof}

\section{étale $\varphi$-modules}

In this section, we prove that étale $\varphi$-modules on prisms does not change by passing to perfections. We first recall the definition of étale $\varphi$-modules.  

\begin{definition}
Let $R$ be a ring equipped with a ring morphism $\varphi : R\rightarrow R$, an étale $\varphi$-modules over $R$ is a finite projective  $R$-module $M$ equipped with an $R$-module isomorphism
\[
F: \varphi^*M=M \otimes_{R,\varphi} R \overset{\sim}{\longrightarrow} M.
\]
The morphisms between étale $\varphi$-modules are $R$-module morphisms preserving $F$. We denote by $\textbf{ÉM}_{/R}$ the category of étale $\varphi$-modules over $R$. 
\end{definition}    

First,  we observe that by passing to naive perfections, the category does not change. 

\begin{proposition} \label{perfectionetale}
There is an equivalence of categories 
\[
\textbf{ÉM}_{/R} \overset{\sim}{\longrightarrow}
\textbf{ÉM}_{/\underset{\varphi}{colim} \ R}
\]
induced by base changing to $\underset{\varphi}{colim} \ R$. 
\end{proposition}

\begin{proof}
For notational convenience, we index the rings $R$ in the relevant system by $R_n$, i.e. the ring $\underset{\varphi}{\text{colim}} \ R$
is the colimit of the cofiltered system
\[
R_0 \overset{\varphi}{\longrightarrow} R_1 \overset{\varphi}{\longrightarrow} R_2 \overset{\varphi}{\longrightarrow} \cdots
\]
with each $R_i = R$. As the data of an étale $\varphi$-module is finite in nature, an étale $\varphi$-module over 
$\underset{\varphi}{\text{colim}} \ R$
comes via base change from an étale $\varphi$-module over $R_n$ for some $n$. We need to show that it has further descent to $R_0$. Now let $M$ be an étale $\varphi$-module over $R_n$. Since $R_n=R=R_0$, we can view $M$ as an étale $\varphi$-module over $R_0$, and we claim that
$M \otimes_{R_o} R_n = (\varphi^{n})^{*}M $
is isomorphic to $M$ as  étale $\varphi$-modules over $R_n$. Iterate the $\varphi$-module structure
$F: \varphi^*M \cong M$, we obtain an $R$-module isomorphism
\[
G := F \circ \varphi^* F \circ \cdots \circ (\varphi^{n-1})^*F : (\varphi^n)^* M \overset{\sim}{\longrightarrow} M
\]
we need to check that this is an étale $\varphi$-module isomorphism, i.e. 
$G \circ (\varphi^n)^*F = F \circ \varphi^*G$, but this is clear. This proves essential surjectivity. 

For fully faithfullness, we need to show that for
$(M,F_M), (N,F_N) \in \textbf{ÉM}_{/R_0}$, an arrow 
$M \otimes_{R_0} \underset{\varphi}{\text{colim}} \ R
\rightarrow 
N\otimes_{R_0} \underset{\varphi}{\text{colim}} \ R$
in
$\textbf{ÉM}_{/\underset{\varphi}{\text{colim}} \ R}$
comes uniquely from an arrow in 
$\textbf{ÉM}_{/R_0}$ 
via base change. The arrow comes from an arrow in 
$\textbf{ÉM}_{/R_n}$
for some $n$ by standard finiteness argument, and we are reduced to showing that any arrow 
\[
f: M\otimes_{R_0} R_n \longrightarrow N \otimes_{R_0}R_n
\]
in $\textbf{ÉM}_{/R_n}$ has a unique descent to $R_0$. We show that it descends uniquely to $R_{n-1}$, which proves the claim by iteration. Now we can assume $n=1$, the previous paragraph shows that 
\[
F_M  : \varphi^*M \overset{\sim}{\longrightarrow} M
\]
\[
F_N 
:\varphi^*N \overset{\sim}{\longrightarrow} N
\]
are isomorphisms as étale $\varphi$-modules over $R$ (being denoted by $G$ in the previous paragraph). Let 
\[
g:= F_N \circ f \circ F_M^{-1} : M \longrightarrow N
\]
then $g$ is an arrow in 
$\textbf{ÉM}_{/R_0}$, and we claim that $\varphi^*g = f$. Since $f$ is an arrow in $\textbf{ÉM}_{/R_1}$, we have 
\[
\varphi^* F_N \circ \varphi^*f = f\circ \varphi^*F_M 
\]
and then 
\[
\varphi^*g= \varphi^*F_N \circ \varphi^*f \circ \varphi^*F_M^{-1} = f \circ \varphi^*F_M \circ \varphi^*F_M^{-1} 
= f,
\]
proving the existence of the descent. It is unique since any descent $h$ of $f$ satisfies the relation $\varphi^*h =f$ by definition, and $h$ satisfies the relation
$h \circ F_M =F_N \circ \varphi^*h $
as it is in 
$\textbf{ÉM}_{/R_0}$, combining the two we have
\[
h\circ F_M = F_N \circ \varphi^*h = F_N \circ f
\]
proving 
\[
h= F_N \circ f \circ F_M^{-1} = g.
\]
\end{proof}

Next we show that $p$-adic completion of the perfection does not lose information of étale $\varphi$-modules over $p$-complete rings. 

\begin{lemma} \label{uiiiiiiiweuwi}
Let $R$ be a $p$-adically complete ring equipped with a ring morphism $\varphi: R\longrightarrow R$, then base change induces an equivalence of categories 
\[
\textbf{ÉM}_{/R} \overset{\sim}{\longrightarrow} \textbf{ÉM}_{/\underset{\varphi}{(colim} \ R)^{\wedge}_p}.
\]
\end{lemma}

\begin{proof}
By $p$-completeness of 
$\underset{\varphi}{(\text{colim}} \ R)^{\wedge}_p$, 
we have
\[
\textbf{ÉM}_{/\underset{\varphi}{(\text{colim}} \ R)^{\wedge}_p} = \underset{n}{\text{lim}} \ \textbf{ÉM}_{/\underset{\varphi}{(\text{colim}} \ R)^{\wedge}_p/p^n} =
\underset{n}{\text{lim}} \ \textbf{ÉM}_{/\underset{\varphi}{(\text{colim}} \ R)/p^n} = 
\]
\[
\underset{n}{\text{lim}} \ \textbf{ÉM}_{/\underset{\varphi}{\text{colim}} \ (R/p^n)}
=
\underset{n}{\text{lim}} \ \textbf{ÉM}_{/ (R/p^n)}
=
\textbf{ÉM}_{/ R}
\]
where we use the commutativity of colimit with tensoring with $\mathbb{Z}/p^n$ in the third equality, proposition \ref{perfectionetale} in the fourth equality and $p$-completeness in the last one. 
\end{proof}

We now specialize to the case of prisms and closely related rings. Let $(A,I)$ be a bounded prism, we want to study étale $\varphi$-modules over $A[\frac{1}{I}]^{\wedge}_p$. There are two natural ways to form a perfection of the ring. The first is take the perfection directly and then $p$-complete it, namely 
\[
(\underset{\phi}{\text{colim}} \ A[\frac{1}{I}]^{\wedge}_p)^{\wedge}_p,
\]
while the second is to take the perfection of the prism $(A,I)$ first, then inverting $I$ and $p$-complete, i.e.
\[
((\underset{\phi}{\text{colim}} \ A )^{\wedge}_{(p,I)} [\frac{1}{I}])^{\wedge}_p.
\]
It is the second one that will ultimately help us, and we want to understand étale $\varphi$-modules over it. Note that we have already understand étale $\varphi$-modules over the first ring, namely lemma \ref{uiiiiiiiweuwi} tells us that étale $\varphi$-modules over 
$(\underset{\phi}{\text{colim}} \ A[\frac{1}{I}]^{\wedge}_p)^{\wedge}_p$
is the same as étale $\varphi$-modules over 
$A[\frac{1}{I}]^{\wedge}_p$. Observe that we have a natural morphism 
\[
(\underset{\phi}{\text{colim}} \ A[\frac{1}{I}]^{\wedge}_p)^{\wedge}_p
\longrightarrow
((\underset{\phi}{\text{colim}} \ A )^{\wedge}_{(p,I)} [\frac{1}{I}])^{\wedge}_p
\]
and we have the following theorem characterizing étale $\varphi$-modules over 
$((\underset{\phi}{\text{colim}} \ A )^{\wedge}_{(p,I)} [\frac{1}{I}])^{\wedge}_p$.

\begin{theorem} \label{934030598}
Let $(A,I)$ be a bounded prism such that $\phi(I) \ \text{mod} \ p$ is generated by a non-zero divisor in $A/p$, then
we have an equivalence of categories 
\[
\textbf{ÉM}_{/(\underset{\phi}{\text{colim}} \ A[\frac{1}{I}]^{\wedge}_p)^{\wedge}_p} \overset{\sim}{\longrightarrow} \textbf{ÉM}_{/((\underset{\phi}{\text{colim}} \ A )^{\wedge}_{(p,I)} [\frac{1}{I}])^{\wedge}_p}
\]
induced by base change.
\end{theorem}

\begin{proof}
We compute 
$(\underset{\phi}{\text{colim}} \ A[\frac{1}{I}]^{\wedge}_p)^{\wedge}_p$
first. Being a $p$-complete perfect $\delta$-ring, we know from \cite{2019arXiv190508229B} corollary 2.31 that 
\[
(\underset{\phi}{\text{colim}} \ A[\frac{1}{I}]^{\wedge}_p)^{\wedge}_p= W((\underset{\phi}{\text{colim}} \ A[\frac{1}{I}]^{\wedge}_p)^{\wedge}_p/p)
=
\]
\[
W((\underset{\phi}{\text{colim}} \ A[\frac{1}{I}]^{\wedge}_p)/p)=
W(\underset{\phi}{\text{colim}} \ (A[\frac{1}{I}]^{\wedge}_p/p))
=
W(\underset{\phi}{\text{colim}} \ (A/p [\frac{1}{I}]))
\]
where we use again the commutation of colimit with tensoring with $\mathbb{Z}/p$. $A$ is $(p,I)$-complete as $(A,I)$ is a bounded prism, so $A/p$ is $I$-adically complete. By \cite{2019arXiv190508229B} lemma 3.6, $\phi(I)A$ is principal, so $\phi(I)  \equiv I^p \ \text{mod} \ p$ is principal which is generated by a non-zero divisor by assumption. It follows that $A/p [\frac{1}{I}]$ is a Tate ring with ring of definition $A/p$. 

On the other hand, 
\[
((\underset{\phi}{\text{colim}} \ A )^{\wedge}_{(p,I)} [\frac{1}{I}])^{\wedge}_p 
=
W((\underset{\phi}{\text{colim}} \ A )^{\wedge}_{(p,I)}/p [\frac{1}{I}])
=
W((\underset{\phi}{\text{colim}} \ A/p )^{\wedge}_{I} [\frac{1}{I}])
\]
by \cite{2019arXiv190508229B} corollary 2.31 again. 

We know that étale $\varphi$-modules over $W(R)$ are equivalent to lisse sheaves on $R$ for a perfect ring $R$ by \cite{kedlaya2015relative} proposition 3.2.7, hence it is enough to compare the finite étale sites of
$(\underset{\phi}{\text{colim}} \ A/p )^{\wedge}_{I} [\frac{1}{I}]$
and
$\underset{\phi}{\text{colim}} \ (A/p [\frac{1}{I}])$. 

We note that 
$(\underset{\phi}{\text{colim}} \ A/p )^{\wedge}_{I} [\frac{1}{I}]$
is the completed perfection of the Tate ring $A/p[\frac{1}{I}]$, and by the following lemma the finite étale site of 
$(\underset{\phi}{\text{colim}} \ A/p )^{\wedge}_{I} [\frac{1}{I}]$
is the same as that of $A/p[\frac{1}{I}]$. But perfection also does not change finite étale site (see \cite{kedlaya2015relative} theorem 3.1.15(a)), hence finite étale site of 
$\underset{\phi}{\text{colim}} \ (A/p [\frac{1}{I}])$
is also identified with 
$A/p[\frac{1}{I}]$. This proves that the finite étale sites of
$(\underset{\phi}{\text{colim}} \ A/p )^{\wedge}_{I} [\frac{1}{I}]$
and
$\underset{\phi}{\text{colim}} \ (A/p [\frac{1}{I}])$
are equivalent via base change (as all intermediate equivalences here are through base change). 
\end{proof} 

\begin{lemma}
Let $R$ be a Banach ring of characteristic $p$ (in the sense of \cite{kedlaya2015relative} definition 2.2.1), then the finite étale site of $(\underset{\phi}{colim} \ R)^{\wedge}$
is equivalent to that of $R$ via base change. 
\end{lemma} 

\begin{proof}
Let $R^u$ be the uniformization of $R$ as defined in \cite{kedlaya2015relative} definition 2.8.13, then by \cite{kedlaya2015relative} proposition 2.8.16, the finite étale site of $R^u$ is equivalent to that of $R$ under base change. Moreover, by \cite{kedlaya2015relative} theorem 3.1.15 (b), the finite étale sites of $(\underset{\phi}{colim} \ R^u)^{\wedge}$
and $R^u$ are equivalent, so we have the comparison between finite étale sites of $R$ and
$(\underset{\phi}{colim} \ R^u)^{\wedge}$. We claim that 
$(\underset{\phi}{colim} \ R^u)^{\wedge} = (\underset{\phi}{colim} \ R)^{\wedge}$. By \cite{kedlaya2015relative} lemma 2.6.2, there exists an affionid system $R_i$ (see \cite{kedlaya2015relative} definition 2.6.1 for definition) such that $R= (\text{colim} \ R_i)^{\wedge}$, then $R^u= (\text{colim} \ R_i^{\text{red}})^{\wedge}$ by \cite{kedlaya2015relative} corollary 2.5.6. Let $R_{\text{perf}}:= \underset{\phi}{colim} \ R$,
then 
\[
(R_{\text{perf}})^{\wedge} = ((\text{colim} \ R_i)_{\text{perf}})^{\wedge}
=
(\text{colim} \ (R_i)_{\text{perf}})^{\wedge} =
(\text{colim} \ (R_i^{\text{red}})_{\text{perf}})^{\wedge} =
(R^u_{\text{perf}})^{\wedge}
\]
where we use that $T^{\text{red}}_{\text{perf}} = T_{\text{perf}}$ 
for any ring $T$ of characteristic $p$, which can be checked directly. 
\end{proof}

Combining all the equivalences we have established, we have the following theorem.

\begin{theorem} \label{perfectionimp}
With assumptions as in theorem \ref{934030598}, we have an equivalence of categories
\[
\textbf{ÉM}_{/A[\frac{1}{I}]^{\wedge}_p} \overset{\sim}{\longrightarrow} \textbf{ÉM}_{/((\underset{\phi}{\text{colim}} \ A )^{\wedge}_{(p,I)} [\frac{1}{I}])^{\wedge}_p}
\]
induced by base change.
\end{theorem}

\section{$(\varphi,\Gamma)$-modules and prismatic $F$-crystals}

In this section, we interpret $(\varphi, \Gamma)$-modules in terms of prismatic $F$-crystals. We recover the equivalence between Galois representations and $(\varphi, \Gamma)$-modules using the new interpretation. 

Let us first recall the definition of $(\varphi, \Gamma)$-modules. Let $K$ be a finite totally ramified extension of $W(k)[\frac{1}{p}]$, and $\bold{A}_K^+$ be as in  section \ref{prismsection}. Recall that $\bold{A}_K \subset W(\mathbb{C}^{\flat})$ is stable by the canonical Frobenius lifting $\phi$ and the action of the Galois group Gal($\bar{K}/K$) on 
$W(\mathcal{O}_{\mathbb{C}}^{\flat})$. Moreover, the action factorizes through 
\[
\Gamma:=\text{Gal}(K(\zeta_{p^{\infty}})/K).
\]

\begin{definition}
A $(\varphi, \Gamma)$-module over $\bold{A}_K$ is an étale $\varphi$-module $M$ over
$\bold{A}_K$ (with respect to the $\phi$-structure on $\bold{A}_K$), i.e. a finite projective $\bold{A}_K$-module $M$ equipped with an isomorphism
\[
F: \phi^*M \cong M,
\]
together with an action of
$\Gamma$
on $M$ that commutes with $F$, and semilinear with respect to the action of $\Gamma$ on $\bold{A}_K$. 
\end{definition}

We have the following theorem.

\begin{theorem} \label{phigamma}
The category of prismatic F-crystals in $\mathcal{O}_{\prism}[\frac{1}{I_{\prism}}]^{\wedge}_p$-modules over $Spf(\mathcal{O}_K)$
is equivalent to the category of $(\varphi, \Gamma)$-modules over 
$\bold{A}_K$.
\end{theorem}

\begin{proof}
By example \ref{examplewww}, we see that a prismatic F-crystal in $\mathcal{O}_{\prism}[\frac{1}{I_{\prism}}]^{\wedge}_p$-modules over $Spf(\mathcal{O}_K)$ is an étale $\varphi$-module over 
$\bold{A}_K$
together with an isomorphism
\[
M\otimes_{\bold{A}_K} C [\frac{1}{J}]^{\wedge}_p
\cong 
C [\frac{1}{J}]^{\wedge}_p 
\otimes_{\bold{A}_K} M
\]
of étale $\varphi$-modules over 
$C [\frac{1}{J}]^{\wedge}_p$
satisfying cocycle conditions, 
with 
\[
(C, J)=
(\bold{A}_K^+ \hat{\otimes}_{W(k)} \bold{A}_K^+
 \{\frac{\omega \otimes 1- 1\otimes\omega}{\phi^n([p]_q) \otimes 1}, \frac{1\otimes\phi^n([p]_q)}{\phi^n([p]_q) \otimes 1} \}^{\wedge}, (\phi^n([p]_q) \otimes 1))
 \]
as in example \ref{examplewww}. Then   theorem \ref{perfectionimp} implies that this is equivalent (via base change) to an étale $\varphi$-module $\mathcal{M}$ over 
\[
(\bold{A}_K^+)_{\text{perf}}[\frac{1}{J}]^{\wedge}_p \cong
A_{\text{inf}}(\mathcal{O}_{K(\zeta_{p^{\infty}})^{\wedge}_p})[\frac{1}{J}]^{\wedge}_p \cong 
W((K(\zeta_{p^{\infty}})^{\wedge}_p)^{\flat})
\]
together with an isomorphism  
\[
\mathcal{M}\otimes_{W((K(\zeta_{p^{\infty}})^{\wedge}_p)^{\flat})} B[\frac{1}{I}]^{\wedge}_p 
\cong
B[\frac{1}{I}]^{\wedge}_p 
\otimes_{W((K(\zeta_{p^{\infty}})^{\wedge}_p)^{\flat})} \mathcal{M}
\]
as étale $\varphi$-modules over 
$B[\frac{1}{I}]^{\wedge}_p$, where 
$(B,I)=(C,J)_{\text{perf}}$. 
Then the lemma below tells us that the descent data is equivalent to an action of $\Gamma$ on $\mathcal{M}$ that is semilinear with respect to the action of $\Gamma$ on 
$W((K(\zeta_{p^{\infty}})^{\wedge}_p)^{\flat})$. As the action of $\Gamma$ is already defined on 
$\bold{A}_K$, 
theorem \ref{perfectionimp} tells us that this is equivalent to an action of
$\Gamma$ on $M$ that is semilinear with respect to the action of $\Gamma$ on
$\bold{A}_K$, which is exactly a $(\varphi, \Gamma)$-module over 
$\bold{A}_K$.
\end{proof}

\begin{lemma}
Let $(B, I)$ be
the perfection of the prism
$ (
C,J)$
in lemma \ref{89999}, then 
\[
B[\frac{1}{I}]^{\wedge}_p \cong C^0(\Gamma, W((K(\zeta_{p^{\infty}})^{\wedge}_p)^{\flat}))
\]
where 
$C^0(\Gamma, W((K(\zeta_{p^{\infty}})^{\wedge}_p)^{\flat}))$
is the ring of continuous functions on $\Gamma$ with values in 
$ W((K(\zeta_{p^{\infty}})^{\wedge}_p)^{\flat})$. 
Moreover, the two structure maps from $W((K(\zeta_{p^{\infty}})^{\wedge}_p)^{\flat})$
to $C^0(\Gamma, W((K(\zeta_{p^{\infty}})^{\wedge}_p)^{\flat}))$
are the obvious constant function map, and the one sending $x \in W((K(\zeta_{p^{\infty}})^{\wedge}_p)^{\flat})$ 
to $\{\gamma \rightarrow \gamma(x) \}$. 
\end{lemma}

\begin{proof}
Let
\[
(F,F^+) :=( K(\zeta_{p^{\infty}})^{\wedge}_p, \mathcal{O}_{K(\zeta_{p^{\infty}})^{\wedge}_p}).
\]
By lemma \ref{89999}, $(B,I)$ is the initial perfect prism in $\text{Spf}(\mathcal{O}_K)_{\prism}$ equipped with two arrows from $(\bold{A}_K^+, (\phi^n([p]_q)))$ into it.  This is the same as the initial perfect prism in $\text{Spf}(\mathcal{O}_K)_{\prism}$ with two arrows from 
\[
(\bold{A}_K^+, (\phi^n([p]_q)))_{\text{perf}} = (A_{\text{inf}}(F^+), \text{Ker}(\theta)),
\]
where we use lemma \ref{9900000000}. Since perfect prisms are equivalent to integral perfectoid rings, we see that $B/I$ is the initial integral perfectoid $\mathcal{O}_K$-algebra with two maps (as $\mathcal{O}_K$-algebras) from 
$F^+$. 
We claim that $B/I[\frac{1}{p}]$ is the initial perfectoid Tate $K$-algebra with two maps from 
$F$. This follows immediately form proposition \ref{integralperfectoid} and proposition \ref{perfectoidtate}. 

Now for any perfectoid Tate $K$-algebra
$A$ with two maps from 
$F$, 
we obtain two maps 
from
$\text{Spa}(A,A^{\circ})$
to 
$\text{Spa}(F,F^+)$, 
as perfectoid spaces over $\text{Spa}(K, \mathcal{O}_K)$ .
We view them as diamonds over $\text{Spd}(K,\mathcal{O}_K)$. Since diamonds is determined by their values on (affinoid) perfectoid test objects,
the previous paragraph shows that 
\[
\text{Spa}(B/I[\frac{1}{p}], B/I[\frac{1}{p}]^+)^{\diamond}
=
\text{Spa}(F,F^+)^{\diamond} \times_{\text{Spd}(K,\mathcal{O}_K)}
\text{Spa}(F,F^+)^{\diamond}
\]
for some ring of definition $B/I[\frac{1}{p}]^+$, which can be identified as the image of $B/I$ in $B/I[\frac{1}{p}]$,
providing we know the right hand side is affinoid perfectoid. But it is well-known that 
\[
\text{Spd}(K,\mathcal{O}_K)=
\text{Spa}(F,F^+)^{\diamond}/\Gamma,
\]
see \cite{scholze2020berkeley} lemma 10.1.7 for example,  
so 
\[
\text{Spa}(F,F^+)^{\diamond} \times_{\text{Spd}(K,\mathcal{O}_K)}
\text{Spa}(F,F^+)^{\diamond} \cong \text{Spa}(F,F^+)^{\diamond} \times \underline{\Gamma}
\cong
\text{Spa}(C^0(\Gamma,F), C^0(\Gamma,F^+))^{\diamond}
\]
where 
$C^0(\Gamma,F)$, resp. $C^0(\Gamma,F^+)$,
is the ring of continuous functions on $\Gamma$ with values in $F$, resp. $F^+$, see \cite{2017arXiv170907343S} example 11.12 for the last isomorphism. We then have
\[
B/I[\frac{1}{p}] \cong  C^0(\Gamma,F).
\]
as the functor from perfectoid spaces over $\text{Spa}(K,\mathcal{O}_K)$ to diamonds over $\text{Spd}(K,\mathcal{O}_K)$ is fully faithful. The two structure maps from $F$ to $C^0(\Gamma,F)$ is then the constant function map, and the one sending $x \in F$ to $\{\gamma \rightarrow \gamma(x)\}$, which is easy to see by chasing through the above canonical isomorphisms. 
Now
\[
B[\frac{1}{I}]^{\wedge}_p=
W((B/I)^{\flat}[\frac{1}{\omega}])=
W((B/I[\frac{1}{p}])^{\flat})
\]
for a uniformizer $\varpi = (a_0, a_1, \cdots) \in (B/I)^{\flat}$
which we can choose so that $a_0 $ divides $p$ in $B/I$. Thus
\[
B[\frac{1}{I}]^{\wedge}_p 
\cong 
W(C^0(\Gamma,F)^{\flat}) \cong
W(C^0(\Gamma,F^{\flat}))
\cong
C^0(\Gamma, W(F^{\flat}))
\]
where the second isomorphism follows directly from the description $F^{\flat} = \lim_{x \rightarrow x^p} F$, and the same descrition for $C^0(\Gamma,F)^{\flat}$. The last isomorphism follows from the concrete description of the Witt vector, namely $W(F^{\flat}) = (F^{\flat})^{\mathbb{N}}$. This clearly commutes with taking continuous functions. Moreover, the ring structure is defined by polynomial equations that is independent of the input ring, whence the last isomorphism. The description of the two structure maps is straightforward by chasing through the various canonical isomorphisms. 
\end{proof}

\begin{remark} \label{ppoqwwqwee}
The action of $\Gamma$ can be directly  detected as follows. For a prismatic F-crystal $M$ over $\mathcal{O}_K$, and $\gamma \in \Gamma$. The action of $\gamma$ on $M((\bold{A}_K^+,(\phi^n([p]_q))))$ 
is induced by the base change isomorphism (the crystal structure)
\[
M((\bold{A}_K^+,(\phi^n([p]_q)))) \otimes_{\bold{A}_K, \gamma}
\bold{A}_K
\cong
M((\bold{A}_K^+,(\phi^n([p]_q))))
\]
corresponding to the arrow
\[
\gamma: (\bold{A}_K^+,(\phi^n([p]_q)))
\overset{\sim}{\longrightarrow }
(\bold{A}_K^+,(\phi^n([p]_q)))
\]
in $(\mathcal{O}_K)_{\prism}$ as described in the proof of corollary \ref{corollaryii}.
\end{remark}

\begin{remark} \label{wqpirei}
The proof also shows that prismatic $F$-crystals in $\mathcal{O}_{\prism}[\frac{1}{I_{\prism}}]^{\wedge}_p$-modules over $\text{Spf}(\mathcal{O}_K)$ are equivalent to $F$-crystals in $\mathcal{O}_{\prism}[\frac{1}{I_{\prism}}]^{\wedge}_p$-modules over $\text{Spf}(\mathcal{O}_K)_{\prism}^{\text{perf}}$, i.e. the site of perfect prisms over $Spf(\mathcal{O}_K)$, and the equivalence is induced by the obvious restriction functor. Indeed,   prismatic $F$-crystals in $\mathcal{O}_{\prism}[\frac{1}{I_{\prism}}]^{\wedge}_p$-modules over $\text{Spf}(\mathcal{O}_K)$ (resp. $\text{Spf}(\mathcal{O}_K)_{\prism}^{\text{perf}}$) are equivalent to $\varphi$-modules over 
$\bold{A}_K^+[\frac{1}{I}]^{\wedge}_p$
together with descent data over $C[\frac{1}{I}]^{\wedge}_p$ (resp. $\varphi$-modules over 
$(\bold{A}_K^+)_{perf}[\frac{1}{I}]^{\wedge}_p$
together with descent data over $C_{perf}[\frac{1}{I}]^{\wedge}_p$), and the proof shows that the latter objects are equivalent. 
\end{remark}

With exactly the same idea, we can recover Galois representations from prismatic F-crystals in $\mathcal{O}_{\prism}[\frac{1}{I_{\prism}}]^{\wedge}_p$-modules. As it is along the same reasoning as above, we only sketch the argument. 

\begin{theorem}
The category of prismatic F-crystals in $\mathcal{O}_{\prism}[\frac{1}{I_{\prism}}]^{\wedge}_p$-modules over $Spf(\mathcal{O}_K)$
is equivalent to the category of  finite free continuous $\mathbb{Z}_p$-representations of $G:= \text{Gal}(\overline{K}/K)$.
\end{theorem}

\begin{proof}
Let
$\mathbb{C}:=\hat{\overline{K}}$
and $\mathcal{O}_{\mathbb{C}}$
be the ring of integers of 
it. By remark \ref{wqpirei}, it enough to work in $\text{Spf}(\mathcal{O}_K)_{\prism}^{\text{perf}}$. 
We can evaluate a prismatic F-crystal at $(A_{\text{inf}}(\mathcal{O}_{\mathbb{C}}), \text{Ker}(\theta))$. 
Let $(B,J)$ be  the product
\[
(A_{\text{inf}}(\mathcal{O}_{\mathbb{C}}), \text{Ker}(\theta)) \times
(A_{\text{inf}}(\mathcal{O}_{\mathbb{C}}), \text{Ker}(\theta))
\]
in $\text{Spf}(\mathcal{O}_K)^{\text{perf}}_{\prism}$, and we need to compute 
\[
B[\frac{1}{J}]^{\wedge}_p = W((B/J[\frac{1}{p}])^{\flat}).
\]
We know that
$B/J[\frac{1}{p}]$ is a perfectoid Tate $K$-algebra and can be interpreted as 
\[
\text{Spa}(B/J[\frac{1}{p}], B/J[\frac{1}{p}]^+)^{\diamond} =  \text{Spd}(\mathbb{C},\mathcal{O}_{\mathbb{C}}) \times_{\text{Spd}(K,\mathcal{O}_K)} 
\text{Spd}(\mathbb{C},\mathcal{O}_{\mathbb{C}}).
\]
We know that 
\[
\text{Spd}(\mathbb{C},\mathcal{O}_{\mathbb{C}}) \times_{\text{Spd}(K,\mathcal{O}_K)} 
\text{Spd}(\mathbb{C},\mathcal{O}_{\mathbb{C}}) \cong 
\text{Spd}(\mathbb{C},\mathcal{O}_{\mathbb{C}}) \times \underline{G} =
\text{Spd}(C^0(G, \mathbb{C}),C^0(G, \mathcal{O}_{\mathbb{C}})),
\]
so 
\[
B[\frac{1}{J}]^{\wedge}_p \cong 
C^0(G, W(\mathbb{C}^{\flat})). 
\]
 With the help of theorem \ref{perfectionimp}, this proves that a prismatic F-crystal is the same as an 
 étale $\varphi$-module over $W(\mathbb{C}^{\flat})$ together with a  $G$-action which is semilinear with respect to the action of $G$ on $W(\mathbb{C}^{\flat})$. Now as $\mathbb{C}^{\flat}$ is algebraically closed, it is well-known that the category of étale $\varphi$-modules over $W(\mathbb{C}^{\flat})$
 is equivalent to the category of finite free $\mathbb{Z}_p$-modules via taking $F$-invariants, see \cite{kedlaya2015relative} proposition 3.2.7 for example. This shows that the prismatic F-crystals are equivalent to finite free $\mathbb{Z}_p$-representations of $G$.
\end{proof}

\begin{remark} \label{wwepiroeio}
Similarly as before, for a prismatic F-crystal $M$, the $G$-action on 
$M((A_{\text{inf}}(\mathcal{O}_{\mathbb{C}}), \text{Ker}(\theta)))$
is induced by the base change isomorphism (the crystal structure)
\[
M((A_{\text{inf}}(\mathcal{O}_{\mathbb{C}}), \text{Ker}(\theta))) \otimes_{W(\mathbb{C}^{\flat}),g} W(\mathbb{C}^{\flat})
\cong
M((A_{\text{inf}}(\mathcal{O}_{\mathbb{C}}), \text{Ker}(\theta)))
\]
with respect to the arrow
\[
g: (A_{\text{inf}}(\mathcal{O}_{\mathbb{C}}), \text{Ker}(\theta))
\overset{\sim}{\longrightarrow}
(A_{\text{inf}}(\mathcal{O}_{\mathbb{C}}), \text{Ker}(\theta))
\]
in $(\mathcal{O}_K)_{\prism}$. 
\end{remark}

\begin{corollary} \label{mxxxxxxx}
The category of $(\varphi, \Gamma)$-modules over $\bold{A}_K$
is equivalent to the category of finite free $\mathbb{Z}_p$-representations of $G = \text{Gal}(\overline{K}/K)$. The equivalence functors are
\[
\mathcal{M} \rightarrow (\mathcal{M} \otimes_{\bold{A}_K} 
W(\mathbb{C}^{\flat}) )^{F\otimes \phi=1}
\]
\[
T \rightarrow 
(T \otimes_{\mathbb{Z}_p} W(\mathbb{C}^{\flat}) )^{H=1} \overset{\text{deperfection}}{\longrightarrow} (\cdot) 
\]
where $\mathcal{M}$ is a $(\varphi, \Gamma)$-module over $\bold{A}_K$, $T$ is a finite free $\mathbb{Z}_p$-representation of $G$, and \[
H:=Gal(\overline{K}/K(\zeta_{p^{\infty}})) \subset G.
\]
The action of  $G$ is diagonal on both 
$T \otimes_{\mathbb{Z}_p} W(\mathbb{C}^{\flat})$
and
$\mathcal{M} \otimes_{\bold{A}_K} 
W(\mathbb{C}^{\flat})$,
where $G$ acts on $\mathcal{M}$ through the canonical quotient $G \rightarrow \Gamma$. 
The $\varphi$-structure on 
$T \otimes_{\mathbb{Z}_p} W(\mathbb{C}^{\flat})$
is defined by $\phi$ on the second factor. Moreover, the deperfection functor is the equivalence from the category of $(\varphi, \Gamma)$-modules over $W(K(\zeta_{p^{\infty}})^{\flat})$
to the category of $(\varphi, \Gamma)$-modules over $\bold{A}_K$, as induced from theorem \ref{perfectionimp}.
\end{corollary}

\begin{proof}
Both categories are equivalent to prismatic F-crystals in  $\mathcal{O}_{\prism}[\frac{1}{I_{\prism}}]^{\wedge}_p$-modules. We check the equivalence functors are given by the stated ones. Given a prismatic F-crystal $M$, the associated $(\varphi, \Gamma)$-module is 
$M((\bold{A}_K^+, (\phi^n([p]_q))))$, while the associated Galois representation is
$M((A_{\text{inf}}(\mathcal{O}_{\mathbb{C}}), \text{Ker}(\theta)))^{\varphi =1}$. 
$M$ being a crystal, we have a canonical identification
\begin{equation} \label{poowooow}
M((A_{\text{inf}}(\mathcal{O}_{\mathbb{C}}), \text{Ker}(\theta))) \cong 
M((\bold{A}_K^+, (\phi^n([p]_q)))) \otimes_{\bold{A}_K}
W(\mathbb{C}^{\flat}) \otimes_{W(\mathbb{C}^{\flat}), \phi^{-n-1}} W(\mathbb{C}^{\flat})
\end{equation}
using the arrow
\[
(\bold{A}_K^+, (\phi^n([p]_q)))
\rightarrow 
(A_{\text{inf}}(\mathcal{O}_{\mathbb{C}}),(\phi^n([p]_q)))
\overset{\phi^{-n-1}}{\longrightarrow}
(A_{\text{inf}}(\mathcal{O}_{\mathbb{C}}), \text{Ker}(\theta)) 
\]
in $\text{Spf}(\mathcal{O}_K)_{\prism}$. We have seen in the proof of proposition \ref{perfectionetale} that the base change along $\phi^{n+1}$ does not affect étale $\varphi$-modules, namely for any étale $\varphi$-module $N$ over $A$ with respect to $\phi:A \rightarrow A$, there is a canonical identification $N\otimes_{A,\phi^{n+1}} A \cong N$
of étale $\varphi$-modules over $A$. Hence we have a canonical idenfication 
\begin{equation} \label{eewwwww}
M((A_{\text{inf}}(\mathcal{O}_{\mathbb{C}}), \text{Ker}(\theta))) \cong 
M((\bold{A}_K^+, (\phi^n([p]_q)))) \otimes_{\bold{A}_K}
W(\mathbb{C}^{\flat})
\end{equation}
of étale $\varphi$-modules over $W(\mathbb{C}^{\flat})$. 
Thus the functor from  
$(\varphi, \Gamma)$-modules to Galois representations is of the expected form, we still need to identify the Galois action. This follows easily from remarks \ref{wwepiroeio} and \ref{ppoqwwqwee}, as the Galois action on both sides of (\ref{poowooow}) are induced by base change, while the identification (\ref{poowooow}) itself is also induced by base change. An easy base change computation together with the observation that the action of $G$ on 
$\bold{A}_K^+$ through the inclusion 
$\bold{A}_K^+ 
\rightarrow A_{\text{inf}}(\mathcal{O}_{\mathbb{C}})$
is via the quotient $G \rightarrow \Gamma$ proves the compatibility of Galois action in (\ref{poowooow}). Then the naturality of the identification $N\otimes_{A,\phi^{n+1}} A \cong N$
gives us the desired description of Galois action on both sides of (\ref{eewwwww}). 

On the other hand, we have a canonical map
\[
(A_{\text{inf}}(\mathcal{O}_{K(\zeta_{p^{\infty}})_p^{\wedge}}), \text{Ker}(\theta)) \longrightarrow
(A_{\text{inf}}(\mathcal{O}_{\mathbb{C}}), \text{Ker}(\theta))/H 
\]
in the site $(\mathcal{O}_K)_{\prism}^{\text{perf}}$, which is not necessarily an isomorphism, but becomes so on the structure sheaf $\mathcal{O}_{\prism}[\frac{1}{I_{\prism}}]^{\wedge}_p$. Indeed, this follows from that $\mathbb{C}$ has (continuous) Galois group $H$ over $K(\zeta_p^{\infty})^{\wedge}_p$ and the tilting equivalence of perfectoid fields. 
This tells us that 
\[
M((A_{\text{inf}}(\mathcal{O}_{K(\zeta_{p^{\infty}})_p^{\wedge}}), \text{Ker}(\theta))) \cong 
M((A_{\text{inf}}(\mathcal{O}_{\mathbb{C}}), \text{Ker}(\theta)))^{H=1}.
\]
Moreover, by looking at the descent data of the $F$-crystal, the $\Gamma$-action on the right hand side is carried to the descent data of the left hand side. 
Then theorem \ref{perfectionimp} applied to $\bold{A}_K$ gives the other direction. Note that the categorical equivalence  in theorem \ref{perfectionimp} implies that base change preserves not only the $\varphi$-structure, but also  $\Gamma$-action, thereby inducing an equivalence of $(\varphi, \Gamma)$-modules. 
\end{proof}

\begin{remark}
The fact that twist by $\phi^{n+1}$ does not change  étale $\varphi$-modules and has all the expected functoriality is used secretly throughout the above proof. For example, it is needed in checking the two functors are quasi-inverse to each other. All that says is that we can ignore the issue caused by twisting by $\phi$. 
\end{remark}

The equivalence functors in the above corollary may look different from the treatment one usually finds in the literature. We now check that they are equivalent. 

Let $\bold{A}$ be the $p$-adic completion of the maximal unramified extension of $\bold{A}_K$ inside $W(\mathbb{C}^{\flat})$, i.e. $\bold{A}$ is the Cohen ring of $\bold{E}_K^{\text{sep}}$, the separable closure of $\bold{E}_K$, which lies inside $W(\mathbb{C}^{\flat})$ and extends $\bold{A}_K$. It is stable by $\phi$ and the action of the Galois group $G$. If we can write $\bold{A}$ as $A[\frac{1}{I}]^{\wedge}_p$ for some prism $(A,I)$, then we can repeat the above argument with $W(\mathbb{C}^{\flat})$ replaced by $\bold{A}$, and deduce the usual description of equivalence functors between Galois representations and $(\varphi,\Gamma)$-modules. However, this is not possible since $\bold{E}_K^{\text{sep}}$
is not complete. Instead, we proceed with the following lemma.

\begin{lemma} \label{llllllllll}
There is an equivalence of categories
\[
\textbf{ÉM}_{/\bold{A}} \overset{\sim}{\longrightarrow}
\textbf{ÉM}_{/W(\mathbb{C}^{\flat})}
\]
induced by base change.
\end{lemma}

\begin{proof}
By \cite{2019arXiv190508229B} corollary 2.31,
\[
(\underset{\phi}{\text{colim}} \ \bold{A})^{\wedge}_p \cong
W(\underset{\phi}{\text{colim}} \ \bold{A}/p)
\cong
W(\underset{\phi}{\text{colim}} \ \bold{E}_K^{\text{sep}})
\cong
W(\bold{E}_K^{\text{alg}}),
\]
then lemma \ref{uiiiiiiiweuwi} implies that 
\[
\textbf{ÉM}_{/\bold{A}} \overset{\sim}{\longrightarrow}
\textbf{ÉM}_{/W(\bold{E}_K^{\text{alg}})}.
\]
As both $\bold{E}_K^{\text{alg}}$ and $\mathbb{C}^{\flat}$ are algebraically closed, étale $\varphi$-modules over $W(\bold{E}_K^{\text{alg}})$
and $W(\mathbb{C}^{\flat})$
are both equivalent to finite free $\mathbb{Z}_p$-modules by taking  $F$-invariants, we have the  equivalence  
\[
\textbf{ÉM}_{/W(\bold{E}_K^{\text{alg}})} \overset{\sim}{\longrightarrow}
\textbf{ÉM}_{/W(\mathbb{C}^{\flat})}.
\]
Combining the two we have the desired equivalence. 
\end{proof}

\begin{theorem}
The category of $(\varphi, \Gamma)$-modules over $\bold{A}_K$
is equivalent to the category of finite free $\mathbb{Z}_p$-representations of $G = \text{Gal}(\overline{K}/K)$. The equivalence functors are
\[
\mathcal{M} \rightarrow (\mathcal{M} \otimes_{\bold{A}_K} 
\bold{A})^{F \otimes \phi=1}
\]
\[
T \rightarrow 
(T \otimes_{\mathbb{Z}_p} \bold{A} )^{H=1} 
\]
where $\mathcal{M}$ is a $(\varphi, \Gamma)$-module over $\bold{A}_K$, $T$ is a finite free $\mathbb{Z}_p$-representation of $G$, and \[
H:=Gal(\overline{K}/K(\zeta_{p^{\infty}})) \subset G.
\]
The action of  $G$ is diagonal on both 
$T \otimes_{\mathbb{Z}_p} \bold{A}$
and
$\mathcal{M} \otimes_{\bold{A}_K} 
\bold{A}$,
where $G$ acts on $\mathcal{M}$ through the canonical quotient $G \rightarrow \Gamma$. 
The $\varphi$-structure on 
$T \otimes_{\mathbb{Z}_p} \bold{A}$ 
is defined by $\phi$
on the second factor. 
\end{theorem}

\begin{proof}
By lemma \ref{llllllllll}, the functor 
\[
\mathcal{M} \rightarrow (\mathcal{M} \otimes_{\bold{A}_K} 
W(\mathbb{C}^{\flat}) )^{F\otimes \phi=1}
\]
in corollary \ref{mxxxxxxx}
is the same as 
\[
\mathcal{M} \rightarrow (\mathcal{M} \otimes_{\bold{A}_K} 
\bold{A})^{F \otimes \phi=1}.
\]
Indeed, as étale $\varphi$-modules over $W(\mathbb{C}^{\flat})$ are all isomorphic to the trivial ones $(W(\mathbb{C}^{\flat})^n, \phi)$, lemma \ref{llllllllll} tells us that étale $\varphi$-modules  over $\bold{A}$ are also of the form $(\bold{A}^n, \phi)$, so taking $F$-invariants produces the same finite free $\mathbb{Z}_p$-modules, namely 
\[
(N \otimes_{\bold{A}} W(\mathbb{C}^{\flat}))^{F=1} \cong N^{F=1}
\]
for any étale $\varphi$-module $N$  over $\bold{A}$.  
This gives the identification of the $\mathbb{Z}_p$-modules, the identification of Galois actions also follows since the equivalence  in lemma \ref{llllllllll} is a categorical one, so arrows corresponding to Galois action are also preserved. 

Conversely, let $\mathcal{N}$ be the image of the functor
\[
T \rightarrow 
(T \otimes_{\mathbb{Z}_p} W(\mathbb{C}^{\flat}) )^{H=1} \overset{\text{deperfection}}{\rightarrow} (\cdot) 
\]
in corollary \ref{mxxxxxxx}. Then by definition, $\mathcal{N} $ is a $(\varphi,\Gamma)$-module over $\bold{A}_K$ such that 
\[
\mathcal{N} \otimes_{\bold{A}_K} W((K(\zeta_{p^{\infty}})^{\wedge})^{\flat})
\cong (T \otimes_{\mathbb{Z}_p} W(\mathbb{C}^{\flat}) )^{H=1}.
\]
Since the functor is the quasi-inverse of
$
\mathcal{M} \rightarrow (\mathcal{M} \otimes_{\bold{A}_K} 
W(\mathbb{C}^{\flat}) )^{F \otimes \phi=1},
$
we have 
\[
\mathcal{N} \otimes_{\bold{A}_K} 
W(\mathbb{C}^{\flat}) \cong
T \otimes_{\mathbb{Z}_p} W(\mathbb{C}^{\flat}),
\]
whence 
\[
(\mathcal{N} \otimes_{\bold{A}_K} \bold{A}) \otimes_{\bold{A}}
W(\mathbb{C}^{\flat}) \cong
(T \otimes_{\mathbb{Z}_p} \bold{A}) \otimes_{\bold{A}} W(\mathbb{C}^{\flat}).
\]
Then lemma \ref{llllllllll} gives us 
\[
\mathcal{N} \otimes_{\bold{A}_K} \bold{A} \cong
T \otimes_{\mathbb{Z}_p} \bold{A},
\]
so
\[
\mathcal{N} \cong (T \otimes_{\mathbb{Z}_p} \bold{A})^{H=1}. 
\]

We have now proved that the functors in corollary \ref{mxxxxxxx} is the same as the ones described in the statement of the theorem, so we can conclude using corollary \ref{mxxxxxxx}. 
\end{proof}

\bibliographystyle{unsrt}
\bibliography{afd}

\end{document}